\documentclass[11pt, english,colorinlistoftodos,oneside,headings=small]{scrartcl}
\usepackage[english]{babel}  
\usepackage{hyperref}
\usepackage{blindtext} 
\usepackage{amsmath} 

\DeclareMathOperator*{\argmin}{arg\,min}
\usepackage{amssymb}
\usepackage{xspace}

\usepackage{calrsfs}
\usepackage{color}
\usepackage{xcolor}
\usepackage[textsize=small]{todonotes}
\usepackage{bm}
\usepackage{cases}
\usepackage{verbatim}
\usepackage{amsthm} 
\usepackage{blindtext} 
\usepackage{listings}
\usepackage{scrlayer-scrpage}
\usepackage{subfigure}
\usepackage{natbib}
\usepackage{graphicx} 
\usepackage{framed, color}
\usepackage{abstract}
\definecolor{shadecolor}{gray}{.65}

\usepackage{appendix}

\usepackage{geometry}
\theoremstyle{definition}
\newtheorem{defi}{Definition}[section]

\newtheorem{rem}[defi]{Remark}
\newtheorem{algo}[defi]{Algorithm}
\newtheorem{thm}[defi]{Theorem}

\newtheorem{lemma}[defi]{Lemma}

\numberwithin{equation}{section}

\begin{document}

\begin{center}
	\Large{\hspace*{0.5cm}Simultaneous Factors Selection and Fusion of Their Levels \\
	in Penalized Logistic Regression\hspace*{0.5cm}} 
	\end{center}
		\vspace*{0.5cm}
	\begin{center}
	Lea Kaufmann$^*$ and Maria Kateri$^{**}$\\
\small{Institute of Statistics, RWTH Aachen University, Germany \\
$^*$ kaufmann@isw.rwth-aachen.de\\
$^{**}$ maria.kateri@rwth-aachen.de}
\end{center}
	\vspace*{1cm}
\begin{abstract}
	Nowadays, several data analysis problems require for complexity reduction, mainly meaning that they target at removing  the non-influential covariates from the model and at delivering a sparse model. When categorical covariates are present, with their levels being dummy coded, the number of parameters included in the model grows rapidly, fact that emphasizes the need for reducing the number of parameters to be estimated. 
In this case, beyond variable selection, sparsity is also achieved through fusion of levels of covariates which do not differentiate significantly in terms of their influence on the response variable. 
In this work a new regularization technique is introduced, called $L_{0}$-Fused Group Lasso ($L_{0}$-FGL) for binary logistic regression. 
It uses a group lasso penalty for factor selection and for the fusion part it applies an $L_{0}$ penalty on the differences among the levels' parameters of a categorical predictor. Using adaptive weights, the adaptive version of $L_{0}$-FGL method is derived. Theoretical properties, such as the existence, $\sqrt{n}$ consistency and oracle properties under certain conditions, are established. In addition, it is shown that even in the diverging case where the number of parameters $p_{n}$ grows with the sample size $n$, $\sqrt{n}$ consistency and a consistency in variable selection result are achieved. Two computational methods, PIRLS and a block coordinate descent (BCD) approach using quasi Newton, are developed and implemented. A simulation study supports that $L_{0}$-FGL shows an outstanding performance, especially in the high dimensional case. 
\end{abstract}

$\sqrt{n}$ consistency, PIRLS algorithm, block coordinate descent (BCD) method

\pagenumbering{gobble}

\pagenumbering{gobble}

\pagenumbering{arabic}
\section{Introduction}

Regularization methods for generalized linear models (GLMs) have been in the center of interest for high-dimensional data analysis, especially  in the last two decades. In this framework, categorical covariates deserve a special attention. First of all, the dummy coding for their levels increases the model complexity rapidly since a categorical predictor with $p+1$ levels brings $p$ predictors into the model. Furthermore, categorical covariates allow dimension reduction not only by eliminating non-significant predictors but also by fusing levels of a predictor that have the same influence on the response. Such fusions lead to sparser models strengthening simultaneously their interpretability and may propose scale adjustments for the categorical predictors. 
Procedures that allow factor selection and levels fusion at the same time are a powerful tool for meaningfully reducing the complexity of the model. 

The most popular model selection and shrinkage estimation method for GLMs is the lasso that uses a $L_1$-type penalty and which was initially proposed for linear regression models (\cite{Tibshirani1996}). It is well-known that the lasso estimator is biased and its model selection can be inconsistent.
An attractive alternative that enjoys selection consistency is the adaptive lasso, proposed by \cite{Zou2006}, which allows
different shrinkage levels for different regression coefficients through the use of adaptive weights. 
Other methods leading to nearly unbiased estimators are the smoothly clipped absolute deviation (SCAD) penalty \cite{FanLi2001} and the minimax concave penalty (MCP) \cite{Zhang2006}.

In a variable selection problem with categorical covariates (factors), the method needs to be applicable factor-wise, i.e., to exclude or include in the model all levels of a factor. For this, a natural and suitable extension of the lasso is the group lasso, originally considered for linear regression  (\cite{KimEtAl2006}, \cite{YuanLin2006}) and later adjusted for logistic regression (\cite{MeierEtAl2008}). A review on group lasso  is provided by \cite{HuangEtAl2012}. The adaptive lasso has also been extended to the adaptive group lasso \cite{Wang2008} while group SCAD \cite{WangEtAl2007} and group MCP \cite{HuangEtAl2012} are the groupwise selection variants of the SCAD and MCP.

However, the above mentioned methods are not able to perform fusion among the levels of a categorical predictor.
Such a fusion can be achieved in a penalized regression framework by applying the penalty on the differences of 
the parameters belonging to the same factor. For the $L_{1}$ penalty this was first considered by \cite{BondellReich2009} in an ANOVA framework and by \cite{GertheissTutz2010_2} for linear regression models. Since $L_{1}$-type penalties lead to biased estimates, penalties with adaptive weights could be considered. However, the performance of such adaptive methods depends on the quality of the adaptive weights used. This fact led \cite{OelkerEtAl2014} to consider the so called $L_{0}$ norm as penalty function on the differences instead. The advantage of this $L_{0}$ based approach is that a $L_{0}$-type penalty just differentiates between an entry (hence a difference) being zero or nonzero and consequently does not depend on the absolute value of the coefficients' differences. A disadvantage of this approach is that the resulting optimization problem is 
non-convex and thus computationally more involved. 
Further, since the $L_0$ norm is not even continuous, it is difficult to investigate theoretical properties. In \cite{OelkerEtAl2014}, the model was fitted with the penalized iteratively reweighted least squares (PIRLS) algorithm while theoretical properties were not in the focus of the paper. This method performs indirectly factor selection, since a factor is excluded  when all parameters corresponding to it are set equal to zero, i.e. to the value of the reference category.
Categories fusion based on penalties imposed on the differences among the parameters of a factor, inducing also factor selection, has been 
considered by \cite{StokellEtAl2020} as well. They intoduced the so called SCOPE methodology, which uses a non-convex penalty, the MCP.

Nevertheless, it is not clear whether such indirect factor selection procedures based on the differences of coefficients from the reference category perform well enough, comparable to a group variable selection approach. 
As the group lasso penalty is a natural choice for factor selection while the above described $L_{0}$ based approach is a convenient choice for levels fusion, this work introduces a new regularization technique, called $L_0$-fused group lasso ($L_0$-FGL), that combines these two penalties for capturing the two different sources of sparsity, namely variable selection and fusion of levels for categorical predictors. 
The use of two penalty functions allows to set the focus on either factor selection or levels fusion, depending on the application context.
Here, $L_0$-FGL is developed and studied in the framework of penalized logistic regression with all covariates being categorical. The method is adjustable to other types of GLMs and cases of co-existence of continuous and categorical covariates. 
We will verify that in our setting the additional group lasso penalty consideration improves the selection performance compared to the approach based solely on the $L_0$ penalties on the differences, which justifies the consideration of an additional penalty term to enforce a stronger factor selection performance.

The rest of the paper is organized as follows. After introducing the new $L_{0}$-FGL method along with its adaptive variant and pointing out its main characteristics in Section 2, the theoretical properties of $L_{0}$-FGL and adaptive $L_{0}$-FGL are investigated in Section 3. In particular, the existence, $\sqrt{n}$-consistency, and existence of an estimate satisfying the asymptotic normality property are proved. Also a result about consistency in variable selection is provided. 
All properties in Section 3 are considered (i) for fixed number of parameters, and (ii) for number of parameters growing in the sample size.
The algorithms used for obtaining the $L_0$-FGL estimates are discussed in Section 4, where also coefficient paths for different computational methods are analyzed. The computational approaches of Section 4 are compared in Section 5 in terms of simulation studies and appropriate goodness of fit measures. A high dimensional design is also included in the simulation studies, which underlines the outstanding performance of the new proposed approach.

\section{The $L_{0}$-fused group lasso for logistic regression}

Consider a binary response $Y$ and $J \in \mathbb{N}$ candidate categorical covariates denoted by $X_{1},...,X_{J}$, observed on a sample of size $n$. In general, some of them could also be continuous but since our goal is to perform variable fusion within the levels of each categorical covariate, we will focus on categorical covariates and neglect the co-existence of continuous covariates. The expansion of the setup and the results for this case is straightforward. Covariate $j \in \{1,...,J\}$ has $p_{j}+1$ levels, coded by $0,\,...,\,p_{j}$, where 0 is chosen to be the reference category. Consequently, our resulting parameter vector is $\bm \beta = (\beta_{0},\bm \beta_{1},...,\bm \beta_{J})^T \in \mathbb{R}^{p+1}$, 
where $p := \sum_{j=1}^J p_{j}$, $\beta_{0}$ denotes the intercept and $\bm \beta_{j}=(\beta_{j1}, \ldots, \beta_{jp_j})$, $j \in \{1,...,J\}$, is the parameter subvector corresponding to the $j$-th factor.

The fixed design matrix is  
$\bm X=(\bm 1, \bm X_{1}^T,...,\bm X_{J}^T)$, with $\bm X_{j}=(X_{j,1}, \ldots,X_{j,p_j})^T	\in \mathbb{R}^{p_{j}},\, \bm 1=(1,...,1) \in \mathbb{R}^n,$ where $X_{j,k}=1$ means that $X_{j}=k$ for $j \in \{1,...,J\}$ and $k \in \{1,...,p_{j}\}$. 
With $\bm x_{i}$ we denote the $i$-th row of the design matrix $\bm X$, hence the $i$-th observation of the covariates $j=1,...,J$. 
The logistic regression model is then given by
\begin{eqnarray}
\mathbb{E}(Y|\bm X= \bm x)= \frac{\exp(\bm x \bm \beta)}{1+\exp(\bm x \bm \beta)}\label{logreg}.
\end{eqnarray}

\subsection{$L_{0}$-Fused Group Lasso}\label{intro}

A penalized regression method minimizes the sum of the log likelihood and an appropriate penalty function. In particular, 
 \begin{eqnarray}
 M_{pen}(\bm \beta)=-L_{n}(\bm \beta)+P_{\lambda}(\bm \beta),	
 \end{eqnarray}
is minimized, where $L_{n}(\bm \beta)$ denotes the log likelihood function and $P_{\lambda}(\bm \beta)$ the penalty function of the chosen method. The penalized regression estimator is then defined as
 \begin{eqnarray}
 \hat{\bm \beta} := \argmin_{\bm \beta \in \mathbb{R}^{p}}M_{pen}(\bm \beta).	
 \end{eqnarray}

For the group lasso (see \cite{YuanLin2006}, \cite{MeierEtAl2008}), $M_{pen}$ and $\hat{\bm \beta}$ become 
	\begin{equation*}
 	P_{\lambda}^{GL}(\bm \beta):=\lambda_{1} \sum_{j=1}^J ||\bm \beta_{j}||_{\bm K_{j}} , \ \text{ and } \ 
  M^{GL}(\bm \beta) :=  -L_{n}(\bm \beta) + P_{\lambda}^{GL}(\bm \beta)  ,
 	\end{equation*}
respectively. Following \cite{YuanLin2006}, for some $\bm \xi \in \mathbb{R}^d$, $d \in \mathbb{N}$ and a  positive definite and symmetric matrix $\bm K \in \mathbb{R}^{d \times d}$, the norm $||\bm \xi ||_{\bm K}$ is defined as 
$||\bm \xi||_{\bm K}:=(\bm \xi^T \bm K \bm \xi)^{\frac{1}{2}}$.

The $L_1$ penalty applied on the differences among the parameters of a factor's levels (see \cite{BondellReich2009}, \cite{GertheissTutz2010_2}) was initially referred as CAS in \cite{BondellReich2009}. Later, \cite{OelkerEtAl2014} considered the $L_0$ penalty for these differences. In a natural way, one can bring up the name CAS-$L_0$ for the corresponding $L_0$ penalty. In the sequel, we will refer to it simply as $L_0$, whenever needed for brevity of notation. In this case it holds
   	\begin{eqnarray*}
 	&& P_{\lambda}^{L_{0}}(\bm \beta)=P_{\lambda}^{CAS-L_{0}}(\bm \beta):=	\lambda_{0} \sum_{j=1}^J \sum_{0 \leq r < s \leq p_{j}} w_{0}^{(j,rs)} ||\beta_{j,r}-\beta_{j,s}||_{0} , \\
&& M^{L_{0}}(\bm \beta)=M^{CAS-L_{0}}(\bm \beta) := -L_{n}(\bm \beta) +P_{\lambda}^{CAS-L_{0}}(\bm \beta).
 	\end{eqnarray*}

To simultaneously perform factor selection and fusion of categories in case of categorical covariates, we propose the following penalty, called $L_{0}$-FGL.
\begin{eqnarray}
P_{\lambda}(\bm \beta):= \lambda_{1} \sum_{j=1}^J ||\bm \beta_{j}||_{K_{j}}+\lambda_{0} \sum_{j=1}^J \sum_{0 \leq r < s \leq p_{j}} w_{0}^{(j,rs)} ||\beta_{j,r}-\beta_{j,s}||_{0},\label{pen1}
\end{eqnarray}
which is an intersection between the well known group lasso and the $L_{0}$ fusion penalty (CAS-$L_{0}$).
In the sequel, we denote by $||\bm t||_{2}$ for some $\bm t \in \mathbb{R}^n$ the euclidean norm $||\bm t||_{2}=\sqrt{\sum_{i=1}^n t_{i}^2}$ while sometimes we write $||\bm t||=||\bm t||_{2}$ for simplicity. 

With regard to the choice of $\bm K_{j}$, we get with $\bm K_{j}= \tilde{w}_{1}^{(j)} {\textbf{1}}_{p_j}$ and $\sqrt{\tilde{w}_{1}^{(j)}}=w_{1}^{(j)}$
\begin{eqnarray}
P_{\lambda}(\bm \beta)= \lambda_{1} \sum_{j=1}^J w_{1}^{(j)}||\bm \beta_{j}||_{2}+\lambda_{0} \sum_{j=1}^J \sum_{0 \leq r < s \leq p_{j}} w_{0}^{(j,rs)} ||\beta_{j,r}-\beta_{j,s}||_{0}.\label{argmin2}
\end{eqnarray}
In particular, we end up with $w_{1}^{(j)}=\sqrt{{p}_{j}}$ using the convenient choice $\bm K_{j}= p_{j}{\textbf{1}}_{p_j}$. 
The use of adaptive weights leads us to the so called $\textit{adaptive } L_{0} \textit{-FGL}$, analogously to the adaptive group lasso. In the following, we will use the latter choice of $\bm K_{j}$ and investigate theoretical properties both for the $L_{0}$-FGL and its adaptive version.

Recall that the $L_{0}$ penalty term of the $L_{0}$-FGL method includes in the sum differences from the reference category $\beta_{j,0}=0$, enforcing thus also factor selection. In this setting, factors selection refers to the case when all categories are fused with the reference category. The $L_{0}$-FGL estimate $\hat{\bm \beta}$ is defined as the minimizer of 
\begin{eqnarray}
M_{pen}(\bm \beta) &:=& \argmin_{\bm \beta \in \mathbb{R}^p} -L_{n}(\bm \beta)+ P_{\lambda}(\bm \beta)	\label{argminL0fgl} \\
&=& \argmin_{\bm \beta \in \mathbb{R}^p} -L_{n}(\bm \beta) + \lambda_{1} \sum_{j=1}^J w_{1}^{(j)}||\bm \beta_{j}||_{2}+\lambda_{0} \sum_{j=1}^J \sum_{0 \leq r < s \leq p_{j}} w_{0}^{(j,rs)} ||\beta_{j,r}-\beta_{j,s}||_{0}. \nonumber
\end{eqnarray}
As already mentioned, the expressions $w_{1}^{(j)}$ for the group lasso part and $w_{0}^{(j,rs)}$ for the $L_{0}$ fusion part are optional weights that allow to put the covariates and their levels on a comparable scale. Using $w_{1}^{(j)}=\sqrt{p_{j}}$ in the group lasso part accounts for the fact that the covariates may have a different number of levels. \\

Figure \ref{L0fgl.constraint} shows the value of the penalty functions for Group Lasso (left), $L_0$ (middle) and $L_0$-FGL (right) for a categorical covariate of three levels, i.e. for $\bm \beta=(0, \beta_1, \beta_2)$. In particular, for $L_0$-FGL the value of $||\bm \beta||_{2}+ ||\beta_{1}-\beta_{2}||_{0}$ for $\beta_{1}, \beta_{2} \in [-2,2]$ and $\lambda_{0}=\lambda_{1}$ is shown. It gets clear that $L_0$-FGL combines both shrinkage and fusion of levels in one penalty. Further, by tuning the tuning parameters $\lambda_0$ and $\lambda_1$ we can put the focus on selection or fusion, depending on the application context.
\begin{figure}[h!] 
\begin{center}
\includegraphics[width=\textwidth]{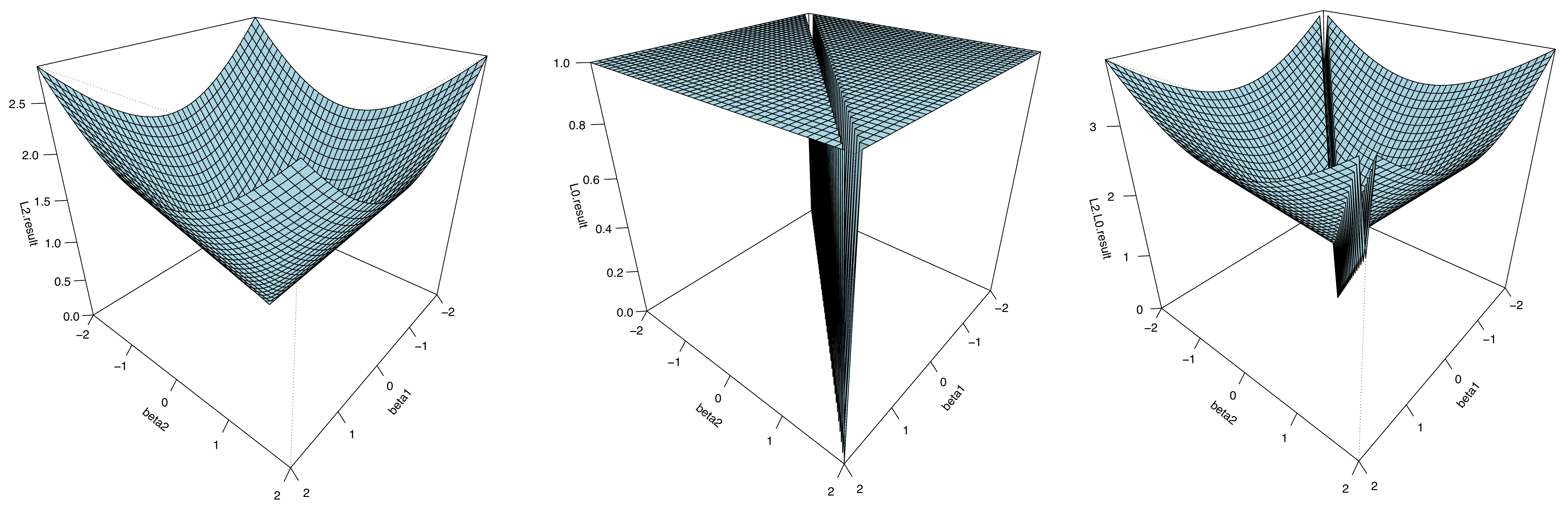}	
\caption{ Visualization of the penalty functions for the group lasso penalty $||\bm \beta||_{2}$ (left), the $L_0$ penalty $||\beta_{1}-\beta_{2}||_{0}$  (middle) and the $L_0$-FGL penalty $||\bm \beta||_{2}+ ||\beta_{1}-\beta_{2}||_{0}$ (right), where all penalties are displayed for $\beta_{1},\beta_{2} \in [-2,2]$ and the tuning parameter were chosen to be equal $\lambda_{0}=\lambda_{1}=1$}
\label{L0fgl.constraint}
\end{center}
\end{figure}

\begin{rem}[On tuning]
	Since $L_0$-FGL has two tuning parameters, a stepwise procedure is proposed for their tuning, according to which  
 \begin{enumerate}
 \item[(1)] $\lambda_{1}^*$ is determined first with cross-validation (CV) setting $\lambda_{0}=0$, and	
 \item[(2)] for fixed $\lambda_{1}=\lambda_{1}^*$, $\lambda_{0}^*$ is determined with CV. 
 \end{enumerate}
 Alternatively, CV could be performed on a two-dimensional grid but this would result to a strongly computational intensive method compared to the procedure above. Due to computational simplicity, we start with the group lasso part 
(tuning of $\lambda_{1}$), since once a factor is excluded from the model, it has not to be investigated for fusion of categories. 

\end{rem}

\section{Existence and Theoretical Properties of $L_{0}$-Fused Group Lasso}
Next, the existence and theoretical properties of $L_{0}$-FGL are investigated, including $\sqrt{n}$ consistency and a theorem about asymptotic normality. Furthermore, consistency in variable selection is analyzed. The case of $p < n $ with $p$ fixed and that of diverging number of parameters, hence $J_{n}$ and consequently $p_{n}$ depending on the sample size $n$ will be considered. 

The next Theorem states the existence of $L_{0}$-FGL. Remark \ref{ex.highdim} in the sequel argues that the existence is also ensured in a high dimensional setup with $p>n$. Notice that, in proving the existence, $p$ is always considered fixed since this is not an asymptotic property.
\begin{thm}[Existence of $L_{0}$-fused group lasso]\label{exL0fgl}
	Let $\lambda_{1} > 0, \, \lambda_{0} \geq 0$ and $0 < \sum_{i=1}^n y_{i} <n$, where $\bm y = (y_1,\ldots, y_n)^T$ with $y_i\in\{0,1\}$, $i\in\{1 \ldots, n\}$ is the vector of observed binary responses. Then, the set
	\begin{eqnarray*}
	S:=\left \{\hat{\bm \beta} \,|\, \hat{\bm \beta} = \argmin_{\bm \beta} -L_{n}(\bm \beta) + \lambda_{1} \sum_{j=1}^J ||\bm \beta_{j}||_{K_{j}}	+ \lambda_{0} \sum_{j=1}^J \sum_{0 \leq r < s \leq p_{j}} w_{0}^{(j,rs)}||\beta_{j,r}-\beta_{j,s}||_{0} \right \}	
	\end{eqnarray*}  
	is nonempty. Moreover, the value of the objective function $M_{pen}(\cdot)$ decreases if coefficients that are close enough to each other are fused. 
	\end{thm}
	\begin{proof}
		see Appendix \ref{app.proof}.
	\end{proof}
\begin{rem}[Existence in high dimensional case $p>n$] \label{ex.highdim} For the $L_{0}$-FGL estimator with $\lambda_{1} >0$ and $\lambda_{0}\geq 0$, the proof above is not restricted to $p \leq n$, hence existence can be ensured in the high dimensional case $p>n$. Recall that a condition for existence when $\lambda_1=0$ or both $\lambda_{1}=\lambda_{0}=0$ is that the maximum likelihood (ML) estimator exists but this is not our focus here since in this case the estimator is not the $L_{0}$-FGL but the CAS-$L_{0}$ or ML, respectively. 
\end{rem}	
For investigating the theoretical properties, some regularity conditions are required, which are provided in Appendix \ref{app.reg1}. We start with a $\sqrt{n}$ consistency result for fixed $p<n$.
\begin{thm}[$\sqrt{n}$ consistency for fixed $p<n$]\label{root-n-consistency}
Let the regularity conditions (Reg1)-(Reg3) from Appendix \ref{app.reg1} hold. Furthermore, assume that $p<n$ is fixed. Set $a_{n}^{1}:=\max \{\lambda_{n}^1w_{1}^{(j)}\,;\, j \leq J\}$ and $a_{n}^{0}:= \max \{\lambda_{n}^{0} w_{0}^{(j,rs)}\,;\, 1\leq r < s \leq p_{j}, j=1,...,J\}$ and assume $a_{n}^{1}/ \sqrt{n} \rightarrow_{P} 0$, $a_{n}^{0} \rightarrow_{P} K$ where $K \in \mathbb{R}$ is some arbitrary constant. Then, it holds that there exists some $L_{0}$-FGL estiamtor $\hat{\bm \beta}$ satisfying $||\hat{\bm \beta}-\bm \beta^*||_{2}=O_{p}\left (\frac{1}{\sqrt{n}}\right )$ for $\hat{\bm \beta}$ the $L_{0}$-FGL estimator. 
\begin{proof}
	see Appendix \ref{app.proof}.
\end{proof}
\end{thm}
Next, the theorem above is extended to the case of non-fixed number of parameters. Hence $J=J_{n}$ depends on $n$ and letting $n\rightarrow \infty$ it may happen that $J_{n}$ and $p_{n}$, respectively, also tend to infinity. 
For proving Theorem \ref{consistency-highdim}, the regularity conditions of Appendix \ref{app.reg1} needs to be slightly modified and are provided  in Appendix \ref{app.reg2}. Even though some parts of the regularity conditions in \ref{app.reg2} correspond to (Reg1)-(Reg3) of \ref{app.reg1}, we will state them independently for adjusting the notation for the case of diverging number of parameters, since in this case several parameters depend on $n$. 
The following theorem shows estimation consistency in the case of a diverging number of parameters, or total number of levels, respectively.
\begin{thm}[Consistency in the diverging case $p_{n} < n$] \label{consistency-highdim}
Let the regularity conditions (div.Reg1)-(div.Reg3) of Appendix \ref{app.reg2} hold. Assume that $J=J_{n}$ hence $J_{n}$, and $p_{n}$ respectively, may grow with the sample size. In addition, let $a_{n}^1$ and $a_{n}^0$ be given analogously to Theorem \ref{root-n-consistency} and require that they exist. With $\alpha_{n}:= \sqrt{\frac{p_{n}}{n}}$ we assume $\alpha_{n}a_{n}^1 J_{n} \rightarrow c_{1}$ and $a_{n}^0 \, p_{n}(p_{n}-1)\rightarrow c_{2}$ for some constants $c_{1},c_{2}>0$ as $n \rightarrow \infty$. Lastly, we assume $\frac{p_{n}^4}{n} \rightarrow 0$. Then, it holds that there exists an $L_{0}$-FGL estimator $\hat{\bm \beta}$ satisfying $||\hat{\bm \beta}_{n}-\bm \beta^*||_{2}=O_{p}(\alpha_{n})$.
\end{thm}
\begin{proof}
	see Appendix \ref{app.proof}.
\end{proof}

\begin{rem}[On the assumptions of Theorem \ref{consistency-highdim}]
The assumption $p_{n}^4/n \rightarrow 0$ implies $p_{n}^2/\sqrt{n} \rightarrow 0$. Hence  the assumption $\alpha_{n}a_{n}^1J_{n}\rightarrow c_{1}$ holds for example if $a_{n}^1$ converges to some constant, since 
\begin{eqnarray*}
	\alpha_{n}a_{n}^1J_{n}=\sqrt{\frac{p_{n}}{n}}J_{n}a_{n}^1\leq \sqrt{\frac{p_{n}}{n}}p_{n}a_{n}^1=\frac{p_{n}^{3/2}}{\sqrt{n}}a_{n}^1\leq \underbrace{\frac{p_{n}^{2}}{\sqrt{n}}}_{\rightarrow 0}a_{n}^1.
\end{eqnarray*}
For the requirement that $a_{n}^0 p_{n}(p_{n}-1) \rightarrow c_{2}$, we observe the case of weights chosen to be constant and equal to one for the $L_{0}$ part. Hence, $a_{n}^0=\lambda_{n}^0$ and we require $\lambda_{n}^0 p_{n}(p_{n}-1) \rightarrow c_{2}$. Since, even in the high dimensional case, it is common to assume  $\lambda_{n}^0 \rightarrow 0$, this assumption is not too restrictive. $\lambda_{n}^0$ does not even have to converge to zero in this case, it is sufficient that it converges to a constant faster than $p_{n}(p_{n}-1)$ converges to infinity. 
The same applies in case of other weights, with $a_{n}^0$ having to converge to a constant faster than $p_{n}(p_{n}-1)$ converges to infinity.
\end{rem}

After having shown the consistency result for the cases of fixed and diverging number of parameters, oracle properties are investigated next.
For this it is required that the true underlying model is sparse, as defined next.

\begin{defi}\label{fisher}
	The true underlying structure is sparse if, without loss of generality,  the true active set $A:= \{j\,|\, \bm \beta_{j}^* \neq \bm 0 \}$ can be written as $A^*=\{1,...,j_{0}\}$ with $j_{0}<J$. In this case, the Fisher information matrix $\bm I_{F}(\bm \beta^*)$ is given in the following form 
\begin{eqnarray}
\bm I_{F}(\bm \beta^*)= \begin{bmatrix}
\bm I_{11} & \bm I_{12} \\
\bm I_{21} & \bm I_{22} 
\end{bmatrix}	\ ,
\end{eqnarray}
where $\bm I_{11} \in \mathbb{R}^{j_{0}\times j_{0}}$.
\end{defi}

\begin{thm}[Existence of estimator satisfying asymptotic normality property for the case fixed $p<n$]\label{asynorm} Assume that (Reg1)-(Reg3) of Appendix \ref{app.reg1} hold and the true underlying structure is sparse (see Definition \ref{fisher}). For the group lasso part we choose the adaptive weights $w_{1}^{(j)}=||\tilde{\bm \beta}_{j}||_{2}^{-\gamma}$ for some arbitrarily chosen $\gamma >0$ where $\tilde{\bm \beta}$ is the unpenalized MLE. Furhtermore, let $\lambda_{n}^1 \cdot n^{-1/2} \rightarrow 0$ and $\lambda_{n}^{1} \cdot n^{(\gamma-1)/2} \rightarrow \infty $. For the tuning of the $L_{0}$ part, we assume $a_{n}^{0}:= \max \{\lambda_{n}^{0} w_{0}^{(j,rs)}\,;\, 1\leq r < s \leq p_{j}, j=1,...,J\}  \rightarrow 0 \,\,(n\rightarrow \infty)$. Then, it holds that there exists an $L_{0}$-FGL estimator $\hat{\bm \beta}= \argmin_{\bm \beta \in \mathbb{R}^p} -L_{n}(\bm \beta)+P_{\lambda}(\bm \beta)$, where $P_{\lambda}(\bm \beta)$ given by (\ref{pen1}), satisfying
\begin{eqnarray*}
\sqrt{n}(\hat{\bm \beta}-\bm \beta^*) \rightarrow_{d} N(0, \bm \Sigma).	
\end{eqnarray*}
Here, $\bm \Sigma=\bm I_{11}^{-1}$.
\end{thm}
\begin{proof}
	see Appendix \ref{app.proof}.
\end{proof}

A consistency result concerning factor selection is discussed next. The desired method should, asymptotically, correctly detect the truly zero parameter vectors as well as the truly nonzero parameter vectors. The theorem below is motivated by the work of \cite{Bunea2008} where the focus lies on $L_{1}$ and $L_{1}+L_{2}$ penalization in linear and logistic regression, ignoring the presence of categorical covariates. Starting from the assumptions needed for the proof of $\sqrt{n}$-consistency of the estimator $\hat{\bm \beta}$, an asymptotic upper bound for the probability $\mathbb{P}(A^* \not \subseteq A_{n})$ is derived, which is a result on the consistency of factor selection of our approach. Recall that $A^*=\{j\,|\,\bm \beta^*_{j} \neq \bm 0\}$ is the active set of the truth and $A_{n}=\{j\,|\, \hat{\bm \beta}_{j}^{(n)}\neq 0\}$ is the active set of the estimate, depending on the sample size $n$.
\begin{thm}[Selection consistency for fixed $p<n$]\label{sel.cons}
Assume that the conditions of Theorem \ref{root-n-consistency} are satisfied and that the true underlying structure is sparse. Then, there exists an $L_{0}$-FGL estimator $\hat{\bm \beta}$ for which it holds that $\forall \varepsilon >0$ there exist a constant $N >0$ such that
\begin{eqnarray}\label{thm3.8}
\mathbb{P}(A^* \not \subseteq A_{n}) < \varepsilon	\,\,\,\,\forall\, n \,\geq\, N.
\end{eqnarray}
\begin{proof}
see Appendix \ref{app.proof}.
\end{proof}
\end{thm}

This result says that, depending on the sample size $n$, there exists an estimator for which the probability that it sets factors to zero which are not truly zero (meaning that we would delete influential factors from our model) can be made arbitrarily small which is a property that is really useful in practice, especially for two-step procedures. 

In the same way, an analogue result for selection consistency in case of a diverging number of parameters can be proved. 

\begin{thm}[Selection consistency in the diverging case $p_{n} < n$]\label{sel.cons.div}
Assume that the conditions of Theorem \ref{consistency-highdim} are satisfied and that the true underlying structure is sparse. Then, there exists an $L_{0}$-FGL estimator $\hat{\bm \beta}$ for which for $\forall \varepsilon >0$ there exists a constant $N >0$ such that (\ref{thm3.8}) holds.
\begin{proof}
see Appendix \ref{app.proof}.
\end{proof}
\end{thm}
\section{Computational Approaches}
Two different computational approaches are considered, the PIRLS algorithm and a block coordinate descent (BCD) procedure.
\subsection{PIRLS Algorithm}\label{sectionpirls}
This approach is suitable for a broad variety of existing penalty functions as discussed in \cite{OelkerTutz2013}. 
In general, PIRLS can be applied to penalty functions of the following form
\begin{eqnarray}
P_{\lambda}^{gen}(\bm \beta)= \sum_{l=1}^L \lambda_{l} p_{l}(||\bm a_{l}^T \bm \beta||_{N_{l}}).\label{pirls1}	
\end{eqnarray}
Here, $L \in \mathbb{N}$ is the number of restrictions with corresponding tuning parameter $\lambda_{l} \geq 0, || \cdot ||_{N_{l}}$ is a semi-norm, or at least some term that makes sense to be used as a penalty, and $p_{l}: \mathbb{R}^+ \rightarrow \mathbb{R}^+$ where $p_{l}(0)=0$ holds and in addition $p_{l}$ is continuously differentiable on $\mathbb{R}^+$ with positive derivative. The vectors $\bm a_{l}^T$ transform the coefficient vector $\bm \beta$, for example in the case of fusion penalties for ordinal data this vector will form the differences of adjacent coefficients (if we have nominal data we will form all pairwise differences). Most of the time, as explained in \cite{OelkerTutz2013}, the penalties are of the form
\begin{eqnarray*}
P_{\lambda}^{gen}(\bm \beta)=	\sum_{j=1}^J \sum_{l=1}^{L_{j}} \lambda_{jl}p_{jl}(||\bm a_{jl}^T \bm \beta_{j}||_{N_{l}})
\end{eqnarray*}
meaning that we penalize each covariate $j \in \{1,...,J\}$ seperately. Keeping this in mind, we will continue to use the more compact way of writing (\ref{pirls1}) where then the parameter $L$ combines both the different number of restrictions and the fact that we may penalize each covariate seperately. Note that here, $p_{l}(\cdot)$ or $p_{jl}(\cdot)$ respectively are functions and do not denote the number of categories of factor $j$ which we also denoted by $p_{j}$.
Further, \cite{OelkerTutz2013} gave an extension to penalties with vector-valued arguments like the group lasso penalty which is also interesting in the case of $L_{0}$-FGL including a group lasso part. Because of the fact that $L_{0}$-FGL has two tuning parameters, we will get two penalty terms in the PIRLS algorithm. In particular, we make the following choices
\begin{eqnarray}
\text{group lasso part:}\,\,\, && p_{l}(\zeta)=\sqrt{p_{l}}\cdot \zeta,\,\, \bm R_{l}\bm \beta= \bm \beta_{j,} \\
\text{$L_{0}$ part:}	\,\,\, && p_{l}(\zeta)= w_{0}^{(j,km)} \zeta ,\,\, \bm a_{l}^T \bm \beta = \beta_{j,k}-\beta_{j,m}\,\, \text{for} \, 0 \leq k < m \leq p_{j}. \label{casL0_1}
\end{eqnarray}
The vectors $\bm a_{l}^T$ are responsible for picking the possible differences corresponding to factor $j$. The entries of these vectors are contained in the set $\{-1,0,1\}$. For the approximations of the "norms", following \cite{OelkerTutz2013}, we have
\begin{eqnarray}
\text{$L_{0}$ Norm:}\,\,\, && N_{l}(\xi)=\frac{2}{1+ \exp(- \gamma |\xi|)}-1\,,\label{casL0_2}\\	
&& D_{l}(\xi)= \frac{2 \gamma}{1+ \exp(-\gamma|\xi|)}\left(1- \frac{1}{1+\exp(-\gamma |\xi|)}\right) \frac{\xi}{\sqrt{\xi^2 +c}}\label{casL0_3}\\
\text{group lasso $||\bm \xi||_{2}$:} && N_{l}(\bm \xi)=(\bm \xi^T \bm \xi +c)^{1/2}.
\end{eqnarray}

\subsubsection*{Coefficient Paths (PIRLS)}
Next we will compare coefficient paths for CAS-$L_{0}$, group lasso and our new approach $L_{0}$-FGL, all computed with the use of the PIRLS algorithm. Consider $J=2$ categorical covariates where $X_{1}$ has $4$ and $X_{2}$ has $3$ categories with equal probabilities. All the covariates are sampled from a multinomial distribution. The true coefficient vector is chosen to be $\bm \beta^*=(2,1.2,1,0.5,-0.8,-0.5)$. We used the $\texttt{simulation}$ function from the package $\texttt{gvcm.cat}$ to simulate our dataset. For the tuning parameters we made the following choices: for $L_{0}$-FGL we chose $\lambda_{max,1}=10$ for the group lasso part and $\lambda_{max,0}=20$ for the $L_{0}$ part. Furthermore, for the CAS-$L_{0}$ estimator we used $\lambda_{max}=10$ and for the group lasso estimator we used $\lambda_{max}=15$. Because of the different maximum tuning values and the fact that $L_{0}$-FGL needs two tuning parameters, the points where the fusion/selection occurs are not comparable among the methods. Further, we chose the unpenalized MLE as starting values. 

The resulting coefficient paths for the chosen approaches are provided in Figure \ref{coefpaths_path_I.1}. We can directly verify that $L_0$-FGL (middle) is an intersection between CAS-$L_{0}$ (right) and the group lasso (left). The huge advantage of $L_0$-FGL compared to the other two is that it combines the ability of factor selection and fusion of categories while the group lasso itself executes factor selection and CAS-$L_{0}$ fusion of categories. Even if the CAS-$L_{0}$ approach is able to select factors since we include reference category zero, we can not be sure of its factor selection performance, the corresponding paths are not smooth. Consequently, the approach of using $L_{0}$-FGL seems to be an advantageous tool that combines both worlds.
\begin{center}
\begin{figure}
\includegraphics[width=\textwidth]{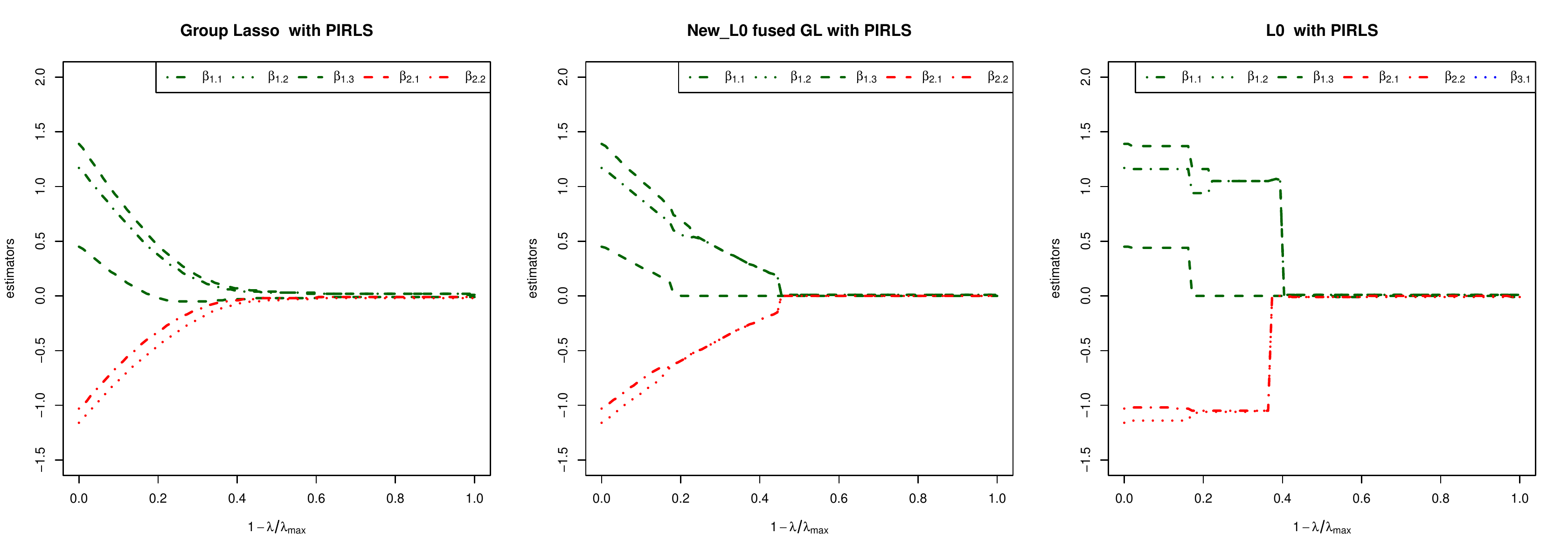}	
\caption{Coefficient paths of two factors with $4$ and $3$ levels, respectively, for group lasso (left), $L_{0}$-FGL (middle) and CAS-$L_{0}$ (right). All methods are computed with PIRLS}
\label{coefpaths_path_I.1}
\end{figure}
\end{center}
\subsection{Block Coordinate Descent}\label{bcd.l0fgl}
A block coordinate descent (BCD) approach with a quasi Newton step for obtaining the estimates is developed as well. The idea is to cycle through the covariates, minimizing with respect to (wrt) one covariate at a time while keeping the others fixed, as for example done in \cite{MeierEtAl2008} and \cite{BrehenyHuang2015}. We start with an approximation of the objective function where the same function as in PIRLS is used for the $L_{0}$ part, whereas the group lasso part is added without approximating it. Notice that an approximation of the group lasso part is not needed, since the  is convex, while in PIRLS the whole penalty is approximated. Details of the approximation of the first part can be found in Appendix \ref{app.approx}.
As explained there, the approximation $g(\bm \beta_{j}, \bm \beta^{(k)})$ is used for the log-likelihood and $L_0$ part of our penalty while the group lasso part is added afterwards.  
The resulting approximation of the $L_0$-FGL penalty function is denoted by 
\begin{eqnarray*}
	\tilde{g}(\bm \beta_{j}, \bm \beta^{(k)}):=g(\bm \beta_{j}, \bm \beta^{(k)}) + \lambda_{1}\sqrt{p_{j}}||\bm \beta_{j}||_{2},\label{approx.gtilde}
\end{eqnarray*}
where $g(\bm \beta_{j}, \bm \beta^{(k)})$ is given by (\ref{defg}).
Function $\tilde{g}(\bm \beta_{j}, \bm \beta^{(k)})$ is minimized wrt $\bm \beta_{j}$ while the rest $\bm \beta_{i},\,i \neq j $ remain fixed. This works because of the separability property of the penalty function which ensures that building the derivative of $\tilde{g}$ wrt $\bm \beta_{j}$ makes the other terms depending on $\bm \beta_{i}$, for $i \neq j$, vanish. The BCD quasi Newton algorithm for $L_{0}$-FGL is thus obtained, as described in Algorithm \ref{algo3_wn} below. 

\begin{algo}[Block Coordinate Gradient descent for $L_{0}$-FGL with Quasi Newton]\hspace{0cm} \label{algo3_wn}\\
\textbf{(1)} Set start value $\bm \beta^{(0)} =0$ if not specified otherwise. Set $k=1$\\
\textbf{(2)} While $|\bm \beta^{(k)}-\bm \beta^{(k-1)}| > \varepsilon$ and $k \leq \text{maxsteps}$\\
\hspace*{1cm} \textbf{(2.1)} Update approximation of $M_{pen}^{CAS-L_{0}}$ including updates of $\bm A_{\lambda}, \widetilde{\bm W}, \tilde{\bm y}$ as they  \\
\hspace*{2cm} depend on the value of the coefficient of the current iteration $\bm \beta^{(k)}$ which gives \\
\hspace*{2cm} the approximation $g(\bm \beta, \bm \beta^{(k)})$. \\
\hspace*{2cm} \textbf{(2.1a)} for $j=1,...,J$\\
\hspace*{4cm} set $\tilde{g}(\bm \beta_{j} , \bm \beta^{(k)}):= g(\bm \beta_{j} , \bm \beta^{(k)})+ \lambda_{1}\sqrt{p_{j} }||\bm \beta_{j} ||_{2}$. \\
\hspace*{4cm} use quasi Newton to obtain $\bm \beta_{j}^{(k+1)}= \argmin_{\bm \beta_{j}} \tilde{g}(\bm \beta_{j}, \bm \beta^{(k)})$ \\
\hspace*{4cm} set $\bm \beta_{j}^{(k+1)}= \bm \beta_{j}^{(t)}$  \hfill (result quasi Newton)\\
\hspace*{2cm} \textbf{(2.1b)} Set $\bm \beta^{(k+1)}=(\bm \beta_{1}^{(k+1)},...,\bm \beta_{j}^{(k+1)}, \bm \beta_{j+1}^{(k)},...,\bm \beta_J^{(k)})$ \\
\hspace*{3.3cm} Set $k=k+1$ \\
\textbf{(3)} Finally, set $\hat{\bm \beta}^{L_{0}-FGL}= \bm \beta^{(k+1)}$.
\end{algo}	

For the execution of the quasi Newton part of this algorithm in our applications we used the function \texttt{optim()} in \texttt{R}.

\subsubsection*{Coefficient Paths (BCD)}
Now we will show resulting coefficient paths for group lasso, CAS-$L_{0}$ and $L_{0}$-FGL where all are computed with the BCD quasi Newton procedure. Assume we have $J=2$ covariates with $p_{1}=p_{2}=3$, hence $4$ levels each, drawn from multinomial distribution with equal probabilities. The true parameter vector was chosen to be given by $\bm \beta^*=(-0.5,2,-1,2,-0.5,-1,1)$. In Figure \ref{coefpaths_path_I.2_Version1} the resulting coefficient paths are displayed. 
\begin{center}
\begin{figure}[h!]
\includegraphics[width=\textwidth]{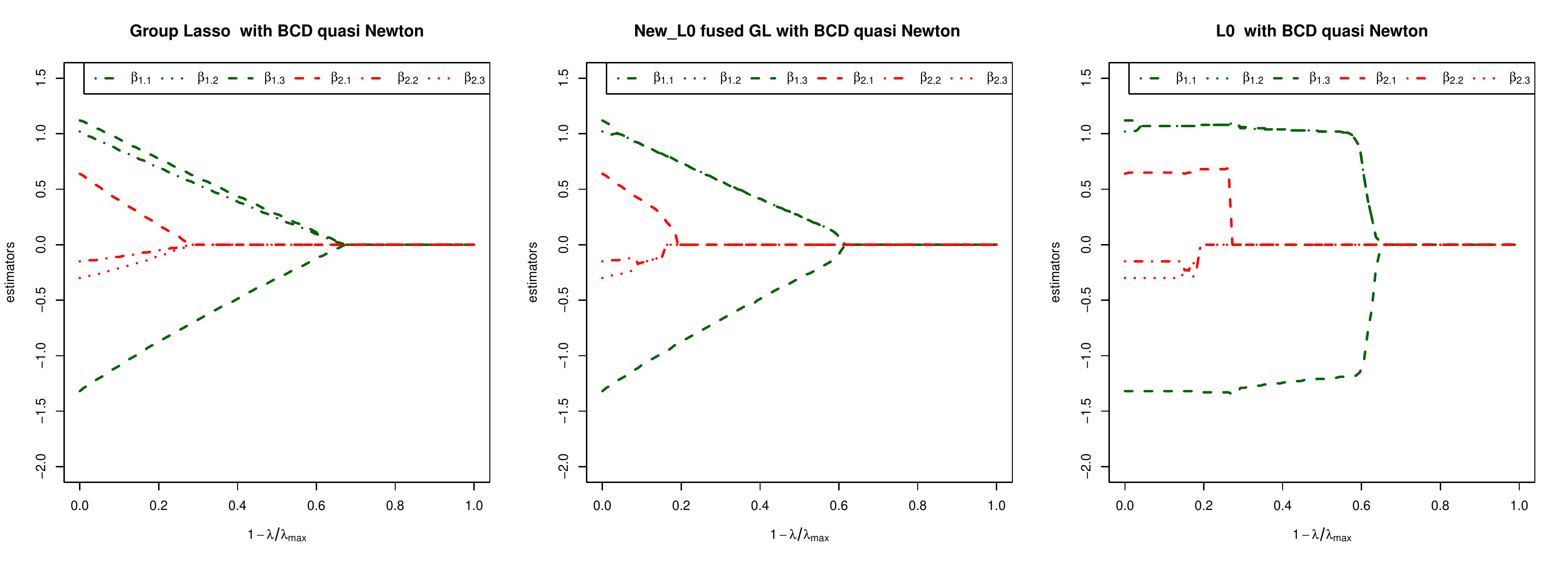}	
\caption{Coefficient paths of two factors with $4$ levels each, for group lasso (left), $L_{0}$-FGL (middle) and CAS-$L_{0}$ (right). All methods are computed with  BCD}
\label{coefpaths_path_I.2_Version1}
\end{figure}
\end{center}
We can see that $L_{0}$-FGL connects the ability of group lasso to select variables and of CAS-$L_{0}$ to fuse coefficients when they are close enough to each other, analogously to the coefficient paths using PIRLS (Section \ref{sectionpirls}). In fact we can see that using BCD and quasi Newton, the paths of group lasso (left) in Figure \ref{coefpaths_path_I.2_Version1} look less smooth than those in Figure \ref{coefpaths_path_I.1} where we used PIRLS. This is caused by the fact that PIRLS uses a quadratic approximation of the penalty function while the BCD quasi Newton approach does not use such an approximation for the group lasso part. To conclude, the method of using a BCD approach with quasi Newton looks promising and its performance will be investigated more detailed in the following simulation studies.

\section{Simulation Studies}
$L_{0}$-FGL procedures, computed with the presented algorithms, are compared with respect to their computational performance in practice for a representative selection of designs. The following methods in both their versions, adaptive and non-adaptive, are included in our comparison.
\begin{enumerate}
	\item[(i)] CAS-$L_{0}$ estimator, computed with PIRLS and the package \texttt{gvcm.cat} 
	\item[(ii)] $L_{0}$-FGL (PIRLS), for which the abbreviation (adaptive) $L_{0}$-FGL PIRLS is used
	\item[(iii)] $L_{0}$-FGL (BCD) and quasi newton, which is referred to as (adptive) $L_{0}$-FGL BCD 
\end{enumerate}

\subsection{Choice of the Weights }\label{section.weights}
Depending on whether the approach under consideration is non-adaptive or adaptive, the corresponding type of weights is used.

\subsubsection*{Non-Adaptive Weights}
Weights can be chosen in the group lasso and the $L_{0}$ part. As already explained, a common choice for the group lasso part is to set $\bm K_{j}$ in such a way that $w_{1}^{(j)}=\sqrt{p_{j}}$. The weights used for the $L_{0}$ fusion part should account for the number of observations per level, and are chosen along the lines of \cite{GertheissTutz2010_2}. Let $n_{j}^{(r)}$ denote the number of observations of level $r$ of the $j$-th covariate, $j\in\{1,\ldots,J\}$. We have to distinguish between the cases of a nominal or ordinal factor, since in the latter one only adjacent categories have to be compared. To sum up, our choice for non-adaptive weights is
\begin{enumerate}
	\item[(i)]  $L_{0}$-part: \\
	$\bullet$ nominal \,\,\ $\displaystyle w_{0}^{(j,rs)}=2(p_{j}+1)^{-1}\sqrt{\frac{n_{j}^{(r)}+n_{j}^{(s)}}{n}}$ , \ \
	$\bullet$ ordinal \,\,\ $\displaystyle w_{0}^{(j,r)}=\sqrt{\frac{n_{j}^{(r)}+n_{j}^{(r-1)}}{n}}$
	\item[(ii)] GL-part: \,\,\ $\displaystyle w_{1}^{(j)}=\sqrt{p_{j}}$
\end{enumerate}
Note that, even for $p>n$, if we assume that for every $j$ the number of levels $p_{j}$ is bounded and in addition $n_{j}^{(r)}/n \rightarrow c_{j}^{(r)} \in (0,1)$ for all $j\in\{1,\ldots,J\}$ and $r\in\{0, \ldots, p_{j}\}$, see \cite{GertheissTutz2010_2}, we can ensure that the weights for the $L_{0}$ part converge to a positive constant.

\subsubsection*{Adaptive Weights}
As for other penalties, we can also use adaptive weights to obtain the adaptive $L_{0}$-FGL method. This is done by choosing in the group lasso part the weights $w_{1}^{(j)}=||\tilde{\bm \beta}_{j}||_{2}^{-1}$. A more general choice is $w_{1}^{(j)}=||\tilde{\bm \beta}_{j}||_{2}^{-\gamma}$ for some chosen $\gamma >0$ as in \cite{Zou2006}, where the oracle properties for this more general choice are proved. To keep the adaptive weights on a comparable scale for the group lasso and the $L_{0}$ part, we multiply the inverse of the norm of the ML estimate with the non-adaptive weight $\sqrt{p_{j}}$. In the $L_{0}$ part we multiply the weights chosen above with the inverse of the difference of the corresponding ML estimates $\tilde{\bm \beta}$, hence we multiply with $|\tilde{\beta}_{j,r}-\tilde{\beta}_{j,s}|^{-1}$. Note that we take here the absolute value of the differences of the ML estimates as in \cite{BondellReich2009}.  Thus, we propose the following adaptive weights
\begin{enumerate}
	\item[(i)]  $L_{0}$-part 
	\begin{itemize}
		\item adaptive nominal \,\,\  $\displaystyle w_{0}^{(j,rs)}=\frac{1}{|\tilde{\beta}_{j,r}-\tilde{\beta}_{j,s}|}\cdot 2(p_{j}+1)^{-1}\sqrt{\frac{n_{j}^{(r)}+n_{j}^{(s)}}{n}}$
		\item adaptive ordinal \,\,\ $\displaystyle  w_{0}^{(j,r)}=\frac{1}{|\tilde{\beta}_{j,r}-\tilde{\beta}_{j,r-1}|}\cdot \sqrt{\frac{n_{j}^{(r)}+n_{j}^{(r-1)}}{n}}$
	\end{itemize}
	\item[(ii)] GL-part adaptive  \,\,\ $\displaystyle w_{1}^{(j)}=\sqrt{p_{j}}||\tilde{\bm \beta}_{j}||_{2}^{-1}$
\end{enumerate}

\subsection{Goodness of Fit Measures}
The approaches under investigation will be compared w.r.t. the following measures
\begin{enumerate}
	\item[(i)] mean squared error coefficients $MSEC(\hat{\bm \beta})= \frac{1}{p} \sum_{j=1}^{p}(\beta_{j}^*-\hat{\beta}_{j})^2$
		\item[(ii)] predictive deviance  $Dev(\bm y, \hat{\bm \mu})= -2 \sum_{i=1}^n \{ y_{i} \log(\hat{\mu}_{i})+(1-y_{i})\log(1-\hat{\mu}_{i})\}$
	\item[(iii)] false positive (FP)/ false negative (FN) rates factor selection 
\begin{eqnarray}
	FP_{s,\text{fac}}(\bm \hat{\bm \beta})= \frac{|\{j \in \{1,...,J\}\,\, : \,\, ||\hat{\bm \beta}_{j}|| \neq 0 \,,\, ||\bm {\beta}^*_{j}||=0\}|}{|\{j \in \{1,...,J\}\,\, :\,\, ||\bm {\beta}_{j}^*|| = 0\}|} \\
 FN_{s,\text{fac}}(\bm \hat{\bm \beta})= \frac{|\{j \in \{1,...,J\}\,\, : \,\, ||\hat{\bm {\beta}}_{j}|| = 0 \,,\, ||\bm{\beta}^*_{j}|| \neq 0\}|}{|\{j \in \{1,...,J\}\,\, :\,\, ||\bm{\beta}^*_{j}|| \neq 0\}|}	
	\end{eqnarray}
	\item[(iv)] FP/FN rates fusion, limited to truly influential factors to ensure that just fusion and no selection is measured
	\begin{eqnarray}
		FP_{f,\text{infl.truth}}= \frac{|\{(j,k,l)\,:\, \hat{\beta}_{j,k} \neq \hat{\beta}_{j,l}, \beta^*_{j,k} = \beta^*_{j,l}\,,\, \left(\sum_{r}|\beta^*_{j,r}|\right ) \neq 0\}|}{|\{(j,k,l): \beta^*_{j,k} = \beta^*_{j,l} \,,\, \left(\sum_{r}|\beta^*_{j,r}|\right ) \neq 0 \}|} \\
  FN_{f,\text{infl.truth}}=\frac{|\{(j,k,l)\,:\, \hat{\beta}_{j,k}= \hat{\beta}_{j,l}, \beta^*_{j,k} \neq \beta^*_{j,l}\,,\, \left(\sum_{r}|\beta^*_{j,r}|\right ) \neq 0\}|}{|\{(j,k,l): \beta^*_{j,k} \neq \beta^*_{j,l} \,,\, \left(\sum_{r}|\beta^*_{j,r}|\right ) \neq 0 \}|}
	\end{eqnarray}
For ordinal factors we compare the adjacent indices $(j,k,k-1)$.
\item[(v)] practical sparsity $|\{j \in \{1,..,J\}: ||\bm \beta_{j}||_{2} \neq 0\}|$ and overall sparsity $|\{ k \in \{1,...,p\}: \beta_{j} \neq 0\}|$
\end{enumerate}

\subsection{Simulation Designs}

To investigate the performance of the approaches discussed above, a design of low and one of high dimension are considered, as described in detail next.
\subsubsection*{Design B8} 
This design is taken from \cite{OelkerEtAl2014}, where the sample size was $n=400$ while we consider $n=1000$. We have 8 ordinal covariates with 4 levels each. Here, 4 covariates are influential and 4 are non-influential. The probabilities for sampling the data where randomly sampled between $0.12$ and $0.44$. The true coefficient vector was chosen to be 
\begin{eqnarray*}
	\bm \beta^*=(2,\,\,\,0,-0.8,-0.8,\,\,\,1,1,0,\,\,\,0.4,0.6,0.8,\,\,\,-0.7,-1,0,\,\,\,0,0,0,\,\,\,0,0,0,\,\,\,0,0,0\,\,\,,0,0,0)^T.
\end{eqnarray*}
The true overall sparsity is $OS^*=9$ and the practical sparsity $PS^*=4$. Hence $50\%$ of the explanatory variables are not influential. 
\subsubsection*{Design highdim}
In this high dimensional design, we observe $60$ ordinal covariates where the first 50 have $4$ categories and the last 10 have $3$ categories each. We draw them from a multinomial distribution with equal probabilities. We chose that just the first five factors are influential, hence just approximately $8\%$ of the covariates have influence on the response. We have $p=171 > n = 100$. In particular, the true coefficient vector was chosen to be
\begin{eqnarray*}
	\bm \beta^*=(2,-1,0.5,2,1.5,1.5,0.5,1,2,2.5,-0.5,-0.3,0.5,2,1,3,0,...,0)^T.
\end{eqnarray*}
The true overall an practical sparsity is given by $OS^*=15$ and $PS^*=5$.

\subsection{Analysis of the Results}

For details on tuning, see Appendix \ref{app.sim}.

\begin{figure}
\begin{center}
	\includegraphics[width=\textwidth]{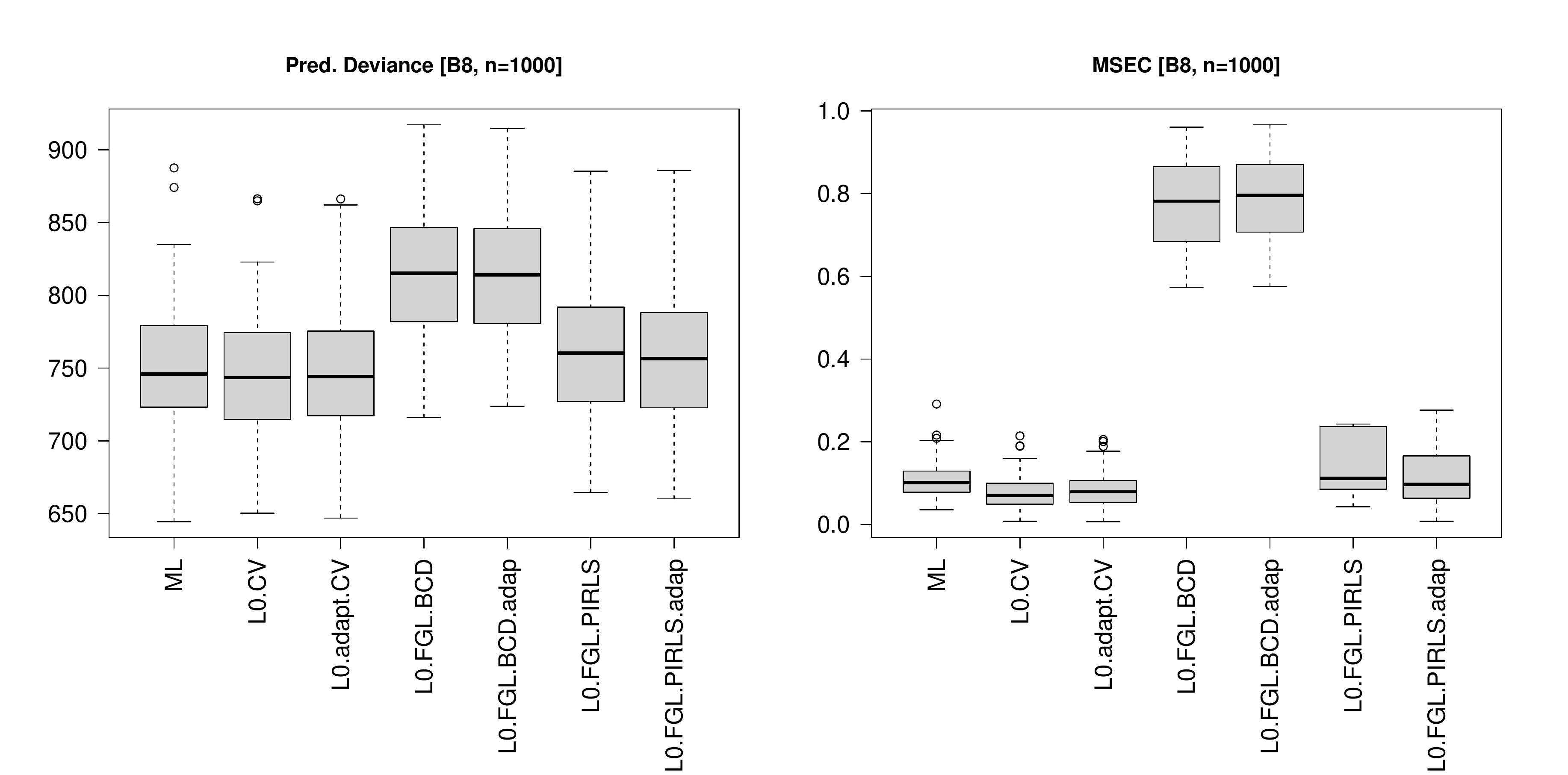}
	\caption{Predictive deviance and MSEC results for design B8 $(n=1000)$}
	\label{B8.Tutz.errorplot.run1}
\end{center}
\end{figure}
Turning our view to the error plots in Figure \ref{B8.Tutz.errorplot.run1}, we can see that in terms of error in the coefficient estimates (MSEC), the approach using $L_{0}$-FGL computed with BCD and quasi Newton shows the worst performance. The other approaches can be allocated on a similar level. Using $n=400$, the resulting errorplots look similar to this one, even though it is not as clear as it is here for a lower sample size.
\begin{table}[h!]
\centering
\small
\begin{tabular}{|r|r|r|r|r|r|r|r|}
  \hline
 & ML & L0.CV & L0.adapt.CV & L0.FGL.& L0.FGL& L0.FGL& L0.FGL. \\ 
  & &  &   & BCD & BCD.adap & PIRLS & PIRLS.adap \\ 
  \hline
  $FP_{s,\text{fac}}$ & 1.00 & 0.62 & 0.45   & 1.00 & 1.00 & 0.40 & 0.18 \\ 
  $FP_{s,\text{fac}}$  & 0.00 & 0.01 & 0.03  & 0.03 & 0.03 & 0.30 & 0.25 \\ 
  $FP_{f,\text{infl.truth}}$ & 1.00 & 0.32 & 0.23   & 0.58 & 0.63 & 0.64 & 0.74 \\ 
  $FN_{f,\text{infl.truth}}$& 0.00 & 0.21 & 0.26   & 0.34 & 0.31 & 0.33 & 0.28 \\ 
   \hline
\end{tabular}
\caption{[B8, $n=1000$] FP/FN rates clustering and selection} 
\label{B8Tutz_run1_n1000_rates}
\end{table}

In terms of FP/FN rates, results are summarized in Table \ref{B8Tutz_run1_n1000_rates}. Focusing on FP/FN factor selection rates we see that the new approach $L_{0}$-FGL (adaptive version, computed with PIRLS) shows an impressively low FP factor selection rate. Of course, this has as a consequence that the corresponding FN rate is higher than for the other approaches. 

Focusing on factor selection, especially in a sparse design, the $L_{0}$-FGL approach shows a really satisfactory performance. The selection performance of $L_{0}$-FGL computed with BCD is not bringing a considerable profit compared to standard ML estimation. Turning our view to FP/FN rates of fusion, we would probably prefer the $L_{0}$ approach compared to $L_{0}$-FGL. In terms of overall and practical sparsity, $L_{0}$-FGL PIRLS outperforms the others, especially the adaptive version.  For $n=400$ we recognized a similar pattern. It is clearly the nearest to the true values while $L_{0}$-FGL BCD selects a less sparse model.
\begin{table}[h!]
\centering
\begin{tabular}{|r|r|r|r|r|r|r|r|}
  \hline
 & ML & L0.CV & L0.adapt.CV &  L0.FGL.& L0.FGL& L0.FGL& L0.FGL. \\ 
  & &  &   & BCD & BCD.adap & PIRLS & PIRLS.adap \\ 
  \hline
OS & 24.00 & 16.11 & 13.93 &   23.66 & 23.67 & 12.16 & 10.61 \\ 
  PS & 8.00 & 6.46 & 5.68 &   7.89 & 7.90 & 4.40 & 3.73 \\ 
   \hline
\end{tabular}
\caption{[B8, $n=1000$] Overall/Practical Sparsity ($OS^*=9, PS^*=4$)} 
\end{table}

To sum up, we can see that $L_{0}$-FGL (PIRLS) results in the most sparse model. Hence it seems to improve the selection performance of the known $L_{0}$ (PIRLS) approach. Finally, $L_{0}$-FGL (BCD) performs worst in this design and seems not to be preferable in this particular setting. But, as we will see, there are designs where this approach outperforms the others.

Next, we analyze the behavior of our approaches in a high dimensional design, which is the most relevant design, since the proposed methods are tailored for reducing the complexity and dimensions in high dimensional designs where the number of predictors exceed the number of sample size 
($p > n$).
 In this design, it is also important to investigate the proportion of replications where the methods fail to yield an estimator. The proportion of fails of all $R=1000$ replications are displayed in table \ref{highdim_run1_fails}. Since $L_{0}$ with PIRLS and both versions 
(adaptive/non-adaptive) of $L_{0}$-FGL with PIRLS fail in all replications, we neglect these approaches in our analysis. Further, notice that also for adaptive $L_{0}$ with PIRLS,  these approaches fail in the majority of replications. Thus, the results have to be interpreted with caution, since the measurements are based on less replications. The only approach that never failed in any of these replications is $L_{0}$-FGL with BCD. The corresponding adaptive version fails in 30\% of the replications which can be explained by the fact that it uses the ML estimate which can cause problems especially in the high dimensional setting. 
\begin{table}[h!]
\centering
\begin{tabular}{|r|r|}
  \hline
 & Proportion of fails \\ 
  \hline
ML & 0.16 \\ 
  L0.CV & 1.00 \\ 
  L0.adapt.CV & 0.87 \\ 
    L0.FGL.PIRLS & 1.00 \\ 
  L0.FGL.BCD & 0.00 \\ 
  L0.FGL.PIRLS.adap & 1.00 \\ 
  L0.FGL.BCD.adap & 0.30 \\ 
   \hline
\end{tabular}
\caption{[highdim] Proportion of fails} 
\label{highdim_run1_fails}
\end{table}

Figure \ref{highdim.errorplot.run1} shows the predictive deviance and MSEC for the different approaches. We can see that $L_{0}$-FGL with BCD shows a lower variability in the MSEC from which we can conclude that it seems to be less sensitive in changes in the data. The same conclusions are derived when observing the predictive deviance. Hence, based on these measurements, for this experimental setup we suggest $L_{0}$-FGL with BCD.
\begin{figure}
\begin{center}
	\includegraphics[width=.8\textwidth]{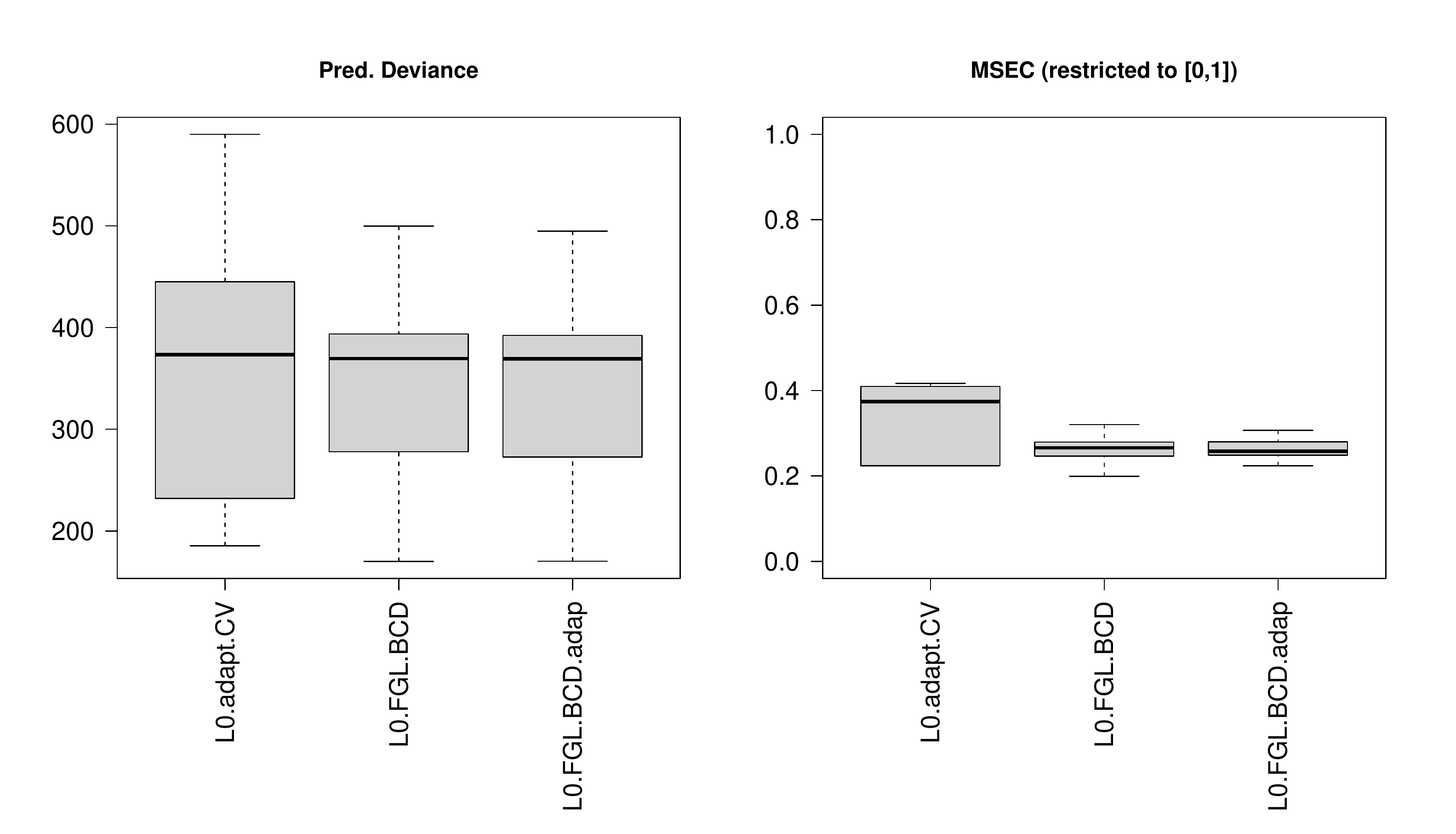}
	\caption{Predictive deviance and MSEC results for design highdim}
	\label{highdim.errorplot.run1}
\end{center}
\end{figure}
\begin{table}[h!]
\centering
\begin{tabular}{|r|r|r|r|r|r|r|r|r|r|}
  \hline
 & ML & L0.CV & L0.adapt.CV & L0.FGL.& L0.FGL& L0.FGL& L0.FGL. \\ 
  &  &  &   & BCD & BCD.adap & PIRLS & PIRLS.adap \\ 
  \hline
OS & 170.00 & - & 15.46   & 60.00 & 66.26 & -  & - \\ 
  PS & 60.00 & - & 10.00  & 24.93 & 27.01 & - & - \\ 
   \hline
\end{tabular}
\caption{[highdim] Overall/Practical Sparsity ($OS^*=15, PS^*=5$)} 
\label{highdim_run1_sparsity}
\end{table}

With respect to overall and practical sparsity (Table \ref{highdim_run1_sparsity}) we observe that $L_{0}$ adaptive selects the most sparse model, but since this approach fails in $87\%$ of the replications, this outcome is not to be trusted. But, we can see that $L_{0}$-FGL BCD (adaptive and non adaptive), which do not fail in the majority of replications, clearly reduce the number of predictors included in the model. Since it is not just important that predictors are excluded from the model but also that the truly non influential ones are excluded, we turn our view to table \ref{highdim_run1_rates}.
\begin{table}[h!]
\centering
\small
\begin{tabular}{|r|r|r|r|r|r|r|r|}
  \hline
 & ML & L0.CV & L0.adapt.CV &  L0.FGL.& L0.FGL& L0.FGL& L0.FGL. \\ 
  & &  &  & BCD & BCD.adap & PIRLS & PIRLS.adap \\ 
  \hline
  $FP_{s,fac}$ & 1.00 & - & 0.17  & 0.41 & 0.45 & - & - \\ 
  $FN_{s,factor}$ & 0.00 & - & 0.83   & 0.50 & 0.50 & - & - \\ 
  $FP_{f,infl.truth}$ & 1.00 & - & 0.08   & 0.23 & 0.27 & - & - \\ 
  $FN_{f,infl.truth}$ & 0.00 &-  & 0.91  & 0.71 & 0.70 & - & - \\ 
   \hline
\end{tabular}
\caption{[highdim] FP/FN rates clustering and selection} 
\label{highdim_run1_rates}
\end{table}

It gets clear that $L_{0}$-FGL with BCD (adaptive and non adaptive) sets approximately 40\% of the truly zero coefficients as nonzero and approximately 50\% of the truly nonzero coefficients are excluded from the model. 
Compared to these results, the low FP factor selection rates of adaptive $L_{0}$ with PIRLS but the high corresponding FN rates are not that satisfactory since, compared with the results on OS and PS, it seems to be the case that this approach sets too many coefficients to zero, so also the few influential ones. But, as already mentioned, the results for the approaches with a high number of fails in the replications, should be carefully interpreted.  

To sum up, the introduced $L_{0}$-FGL procedure, computed with the BCD approach using quasi Newton, is very convenient for such a high dimensional design since it highly reduces the complexity of the problem. It is remarkable that it does not fail in any of the replications even if the number of predictors highly exceeds the sample size.

\section{Conclusion}
In this work, a new approach is introduced, $L_{0}$-FGL,  which performs both, factor selection and levels fusion of categorical predictors, combining two penalty terms, one for selection and one for fusion. Having proven the existence, it is further shown that, under certain regularity conditions, $L_{0}$-FGL satisfies $\sqrt{n}$ consistency, even when the number of parameters grows with the sample size. 
In addition, $L_{0}$-FGL satisfies a result concerning consistency in variable selection in both cases, for fixed or sample size dependent number of parameters. Fixing $p<n$, there exists an adaptive $L_{0}$-FGL estimator satisfying asymptotic normality. 
These properties build a theoretical basis that makes $L_{0}$-FGL an attractive tool.  Simulation studies verified that the new $L_{0}$-FGL approach implemented with PIRLS shows a superior performance in lower dimensional designs and tends to improve the selection performance of the classical $L_{0}$ (PIRLS) method due to the incorperation of the group lasso part. The performance of $L_{0}$-FGL computed with BCD in high dimensions outperformed  the other approaches in the grand majority of replications. It is capable of identifying sparse models and reduces further the dimension of the problem through possible levels' fusion of categorical predictors, delivering thus sound interpretations for the associated effects on the response variable. The theoretical properties along with the simulation results make $L_{0}$-FGL a promising method for modeling high dimensional data with categorical covariates, where sparsity is achieved not only through variable selection but also through levels' fusion. Approaches for goodness of fit testing for models estimated by $L_0$-FGL need to be developed and are currently under research.

\pagestyle{empty}
\bibliography{PaperDraft_L0fusedGL_final.bib}
\clearpage


\pagestyle{plain}
\appendix
\appendixpage
\setcounter{page}{1}

\section{Proofs}\label{app.proof}

\subsection{Regularity conditions (fixed case)}\label{app.reg1}
\begin{enumerate}
\item[(Reg1)] The density can be written as $f(\bm v,\beta)=h(y) \exp(y \bm x\bm \beta - \varphi(\bm x\bm \beta))$ for an observation $\bm v =(y,\bm x)\in \mathbb{R}^{p+2}$. \\ For logistic regression we get $h(y)=$ $n \choose y$, $\varphi(\bm\eta)=n \ln (1+\exp(\bm\eta))$.
\item[(Reg2)] The fisher information matrix $\bm I_{F}(\bm \beta)$ is finite and positive definite in $\bm \beta=\bm \beta^*$ where it holds using (Reg1) that $\bm I_{F}(\bm \beta^*)= \mathbb{E}(\varphi''(\bm x \bm \beta^*)\bm x \bm x^T)$
\item[(Reg3)] There exists an open set $\mathcal{O}$ (depending on $n$) where $\bm \beta^* \in \mathcal{O}$ and for all $\bm \beta \in \mathcal{O}$ there exists $M(\bm x) \in \mathbb{R}$ such that the following holds
\begin{eqnarray*}
&& |\varphi'''(\bm x \bm \beta)| \leq M(\bm x) < \infty \\
&& \mathbb{E}(M(\bm x)|x_{j}x_{k}x_{l}|) < \infty \, \, \, \forall 1 \leq j,k,l \leq p
\end{eqnarray*}
\end{enumerate}
(Reg1) holds actually automatically, by the design of the logistic regression model. However it is provided for the sake of completeness.

Note that the derivatives of $\varphi$ are w.r.t $\bm \beta$. As mentioned in the Appendix of \cite{FanLi2001}, these regularity conditions ensure the asymptotic normality of the ML estimators. The equation in (Reg2) for the Fisher Information matrix can directly be seen if we take the logarithm of $f(y|\bm x,\beta)$, hence the log likelihood, and derive twice w.r.t. $\bm \beta$. So actually, this equation is not a condition but an important equation that we will need during the proof of $\sqrt{n}$ consistency.

\subsection{Regularity conditions (diverging case)}\label{app.reg2}
For the diverging case, assume the following regularity conditions which are stated in \cite{FanLi2001} and also in \cite{FanPeng2004}. Note that in particular also the log likelihood depends on $\bm x$ hence $L_{n}(\bm\beta)=L_{n}(\bm v,\bm \beta)$, where $\bm v=(y,\bm x)$, although we mostly leave out $\bm v$ in the log likelihood for simplicity of notation.
\begin{enumerate}
\item[(div.Reg1)] Assume that the observations $\bm v_{i}=(y_{i},\bm x_{i})\in \mathbb{R}^{p_{n}+1}, i=1,...,n,$ are iid with probability density $f_{n}(\bm v_{1},\bm \beta)$ and we assume that we can write $f_{n}(\bm v_{i},\bm \beta)=h(y_{i}) \exp(y_{i} \bm x_{i}\bm \beta - \varphi(\bm x_{i}\bm \beta))$, as in (Reg1). 
\item[(div.Reg2)] The Fisher information matrix is finite for all $\bm \beta$ and is positive definite at $\bm \beta=\bm \beta^*$.
\item[(div.Reg3)] There exists an open set $\mathcal{O}$ (depending on $n$) for which it holds that $\bm \beta^* \in \mathcal{O}$ such that for all $\bm \beta \in \mathcal{O}$ there exists functions $M_{n,j,k,l}(\bm v) \in \mathbb{R}$ for which it holds 
\begin{eqnarray*}
	\frac{\partial  \log(f_{n}(\bm v_{i},\bm \beta))}{\partial\beta_{j}\partial\beta_{k}\partial\beta_{l}}\leq M_{n,j,k,l}(\bm v_{i})\,\,\forall \bm \beta \in \mathcal{O} \, \, \text{and}\,\,  \forall \,j,k,l =1,...,p_{n}
	\end{eqnarray*}
 Additionally, we assume that for some constant $C_{5}<\infty$ it holds
	\begin{eqnarray*}
	E_{\bm \beta}(M_{n,j,k,l}(\bm v_{i}))<C_{5}<\infty \,\,  \forall \,j,k,l =1,...,p_{n}.
\end{eqnarray*}
\end{enumerate}

\begin{rem}
The above regularity conditions and their consequences are discussed next.
\begin{enumerate}
	\item Using (div.Reg1) we directly get
	 \begin{eqnarray*}
	&& E_{\bm \beta}\left(\frac{\partial \log(f_{n}(\bm v_{1},\bm \beta))}{\partial \beta_{j}}\right)=0\,\,\,\forall j=1,...,p_{n},\\
	&& \bm [I_{F}(\bm \beta)]_{j,k}=E_{\bm \beta}\left(\frac{\partial \log(f_{n}(\bm v_{1},\bm \beta))}{\partial \beta_{j}}\frac{\partial \log(f_{n}(\bm v_{1},\bm \beta))}{\partial \beta_{k}}\right)=E_{\bm \beta}\left(-\frac{\partial^2  \log(f_{n}(\bm v_{1},\bm \beta))}{\partial \beta_{j}\partial \beta_{k}}\right),
\end{eqnarray*}
where $\bm [I_{F}(\bm \beta)]_{j,k}$ denotes the entry of the Fisher information matrix in row $i$ and column $k$ and $\bm v_{i}=(y_{i},\bm x_{i})$.
\item Alternatively to (div.Reg2) we could have also assumed that all the eigenvalues of the Fisher information matrix are finite and strictly positive which ensures the positive definite property, see \cite{FanPeng2004}
\item The fact that we assumed in (div.Reg2) that the Fisher information matrix is finite means in particular that we have $\bm [I_{F}(\bm \beta)]_{j,k}^2<C_{3} < \infty \, \, \forall \,j,k=1,...,p_{n}$  and 
\begin{eqnarray*}
\bm [I_{F}(\bm \beta)]_{j,k}=\displaystyle E_{\bm \beta}\left(-\frac{\partial^2  \log(f_{n}(\bm v_{1},\bm \beta))}{\partial \beta_{j}\partial \beta_{k}}\right) <C_{4}	
\end{eqnarray*}
for some constants $C_{3},C_{4}< \infty$
	
\end{enumerate}
\end{rem}

\begin{rem}
	Note that in (Reg3) $M$ is given as a function of $(\bm x)$ (instead of $\bm v=(y,\bm x)$). In (div.Reg3) $M$ could also be given as a function of  $(\bm x)$, since in both cases, using (div.Reg1) or (Reg1), respectively, the third derivative of the log likelihood does not depend on $y$. 
	\end{rem}

\subsection{Proof of Theorem \ref{exL0fgl}}

\textbf{(1) $S \neq 0$ :}
	Set $J=1$, the proof for $J>1$ works analogously. We will show that the group lasso estimator $\hat{\bm \beta}^{GL} \in S$. By assumption, $0 < \sum_{i=1}^n y_{i} <n $ and by \cite{MeierEtAl2008} (Lemma 1) we can follow that the group lasso estimator $\hat{\bm \beta}^{GL}$ exists. 
	In particular, it holds that there exists an $\bm \varepsilon$-neighborhood of $\hat{\bm \beta}^{GL}$, where $\bm \varepsilon=(\varepsilon_{1},...,\varepsilon_{p}) \in \mathbb{R}^{p}$, such that $\hat{\bm \beta}^{GL}$ minimizes the sum $-L_{n}(\cdot)+\lambda_{1}||\cdot||_{K}$ by definition of the GL estimator. Hence
\begin{eqnarray*}
	-L_{n}(\hat{\bm\beta}^{GL}+\bm \varepsilon)+\lambda_{1}||\hat{\bm \beta}^{GL}+\bm \varepsilon||_{K} \geq -L_{n}(\hat{\bm\beta}^{GL})+\lambda_{1}||\hat{\bm \beta}^{GL}||_{K}.
\end{eqnarray*}
Consequently, adding $\lambda_{0} \sum_{r,s} w_{0}^{(rs)}||\hat{\beta}^{GL}_{r}-\hat{\beta}^{GL}_{s}+\varepsilon_{r}-\varepsilon_{s}||_{0}$ on both sides of the inequality
 \begin{eqnarray}
 M_{pen}(\hat{\bm \beta}^{GL}+\varepsilon)&=&-L_{n}(\hat{\bm \beta}^{GL}+\varepsilon)+\lambda_{1}||\hat{\bm \beta}^{GL}+\varepsilon||_{K}+\lambda_{0} \sum_{r,s} w_{0}^{(rs)}||\hat{\beta}^{GL}_{r}-\hat{\beta}^{GL}_{s}+\varepsilon_{r}-\varepsilon_{s}||_{0}	\nonumber \\
 &\geq & -L_{n}(\hat{\bm \beta}^{GL})+\lambda_{1}||\hat{\bm \beta}^{GL}||_{K}+\lambda_{0} \sum_{r,s} w_{0}^{(rs)}||\hat{\beta}^{GL}_{r}-\hat{\beta}^{GL}_{s}+\varepsilon_{r}-\varepsilon_{s}||_{0}\label{min1}.
 \end{eqnarray}
For the group lasso estimate we have either $\hat{\bm \beta}^{GL}=\bm 0$ or $\hat{\beta}^{GL}_{r} \neq \hat{\beta}^{GL}_{s}\,\, \forall \,r,s$. For the case that we have $\hat{\beta}^{GL}_{r} \neq \hat{\beta}^{GL}_{s}\,\, \forall \,r,s$ we can choose $\bm \varepsilon$ small enough such that $\hat{\beta}^{GL}_{r}+\varepsilon_{r} \neq \hat{\beta}^{GL}_{s} + \varepsilon_{s}\,\, \forall \,r,s$. 
 Consequently, we conclude that the $L_{0}$ norms of $\hat{\beta}^{GL}_{r}-\hat{\beta}^{GL}_{s}$ and $\hat{\beta}^{GL}_{r}-\hat{\beta}^{GL}_{s}+\varepsilon_{r}-\varepsilon_{s}$ coincide since all values of the differences are nonzero. Hence, $||\hat{\beta}^{GL}_{r}-\hat{\beta}^{GL}_{s}+\varepsilon_{r}-\varepsilon_{s}||_{0}=||\hat{\beta}^{GL}_{r}-\hat{\beta}^{GL}_{s}||_{0}$. Then we obtain that the right handside of (\ref{min1}) equals
 \begin{eqnarray*}
  -L_{n}(\hat{\bm \beta}^{GL})+\lambda_{1}||\hat{\bm \beta}^{GL}||_{K}+\lambda_{0} \sum_{r,s} w_{0}^{(rs)}||\hat{\beta}^{GL}_{r}-\hat{\beta}^{GL}_{s}||_{0} = M_{pen}(\hat{\bm \beta}^{GL})
 \end{eqnarray*}
and thus consequently $M_{pen}(\hat{\bm \beta}^{GL}+\bm \varepsilon) \geq M_{pen}(\hat{\bm \beta}^{GL})$ for a sufficiently small $\bm \varepsilon$. If $\hat{\bm \beta}^{GL} = \bm 0$ we get with the same arguments as above
\begin{eqnarray*}
M_{pen}(\hat{\bm \beta}^{GL}+\varepsilon) &\geq& 	-L_{n}(\bm 0)+\lambda_{1}||\bm 0||_{K}+\lambda_{0} \sum_{r,s} w_{0}^{(rs)}\underbrace{||\varepsilon_{r}-\varepsilon_{s}||_{0}}_{\geq \, 0} \\
&\geq& -L_{n}(\bm 0)+\lambda_{1}||\bm 0||_{K}=M_{pen}(\bm 0)=M_{pen}(\hat{\bm \beta}^{GL})
\end{eqnarray*}
thus $M_{pen}(\hat{\bm \beta}^{GL}+\varepsilon) \geq M_{pen}(\hat{\bm \beta}^{GL})$. Hence, the group lasso estimator is an element of the set $S$ giving us that $S \neq \emptyset$ and the first part of the claim follows.  \\
\textbf{(2)} $M_{pen}(\cdot)$ decreases if coefficients that are close enough to each other are fused:
 as we know that the group lasso estimator is one solution of $L_{0}$-FGL but without fusion, we have to show that the objective function decreases if fusion occurs, resulting in an advantage compared to the classical group lasso approach.
	Again, we assume that $J=1$ and we start with the case of an ordinal covariate so we compare adjacent categories for fusion. The goal is to show that the objective function $M_{pen}(\cdot)$ decreases if coefficients that are close enough to each other are fused. Note that, since we chose reference category zero, there is no appearance of the reference category in the coefficient vector $\bm \beta$. Write $\bm \beta_{nf}$ (not fused), $\bm \beta_{f}$ (fused) $\in \mathbb{R}^{p}$ with 
	\begin{eqnarray*}
	&&\bm \beta_{nf}=(\beta_{nf,1},...,\beta_{nf,p})	\,,\,\text{where}\,\, \beta_{nf,i}\neq \beta_{nf,i-1} \, \forall i=2,...,p\,\,\, \text{(not fused),}\\
		&&\bm \beta_{f}=(\beta_{f,1},...,\beta_{f,p})	\,,\,\text{where}\,\, \beta_{nf,i}=\beta_{f,i}\neq \beta_{f,i-1}=\beta_{nf,i-1} \, \forall i=2,...,r-1,r+1,...p\,\, \\ && \hspace*{10cm}\text{and}\,\, \beta_{f,r}=\beta_{f,r-1}.
	\end{eqnarray*}
So in $\bm \beta_{f}$ the categories $r$ and $r-1$ are fused and except for these categories, $\bm \beta_{nf}$ and $\bm \beta_{f}$ coincide. Note that $\beta_{f,r}=\beta_{f,r-1} \in [\min\{\beta_{nf,r},\beta_{nf,r-1}\}, \max\{\beta_{nf,r},\beta_{nf,r-1}\}]$. Without loss of generality, we assume $\min\{\beta_{nf,r},\beta_{nf,r-1}\}=\beta_{nf,r-1}$. Since we observe an ordinal covariate, this holds by definition but observing nominal covariates one has to differentiate between these two cases but the other case works in the same way. Thus it holds that $\beta_{nf,r}-\beta_{nf,r-1}= \epsilon_{1}+\epsilon_{2}=\epsilon > 0$ for some (small) $\epsilon$, meaning that the coefficients of these two categories are close to each other and $\beta_{nf,r}=\beta_{nf,r-1}+\epsilon$, see Figure \ref{existence.fig}. 

\begin{figure}
\begin{center}
\includegraphics[width=0.7\textwidth]{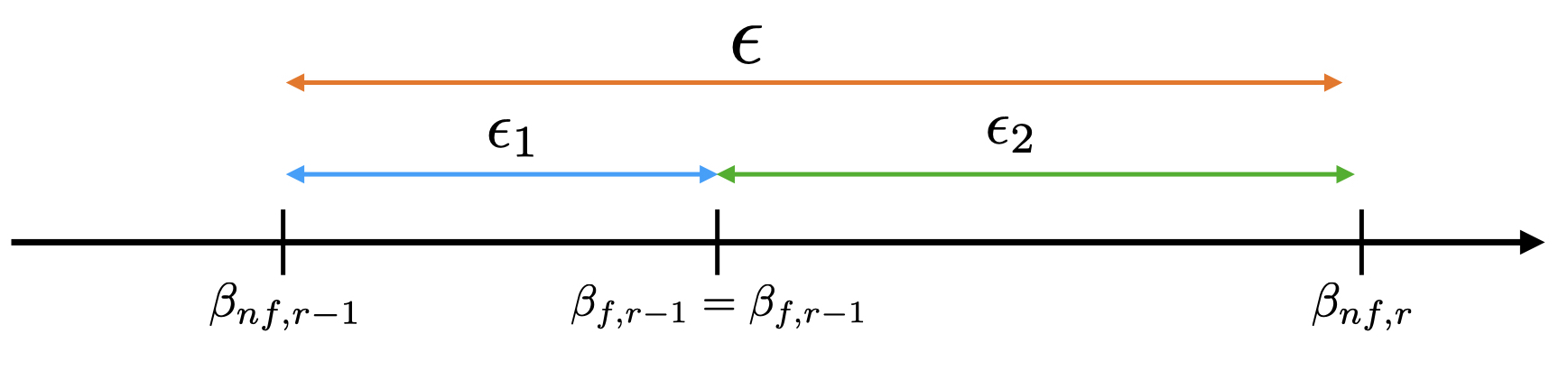}	
\caption{Location of $\beta_{nf,r},\,\beta_{nf,r-1}$ and the fused coefficients $\beta_{f,r}=\beta_{f,r-1}$ for the case $\min\{\beta_{nf,r},\beta_{nf,r-1}\}=\beta_{nf,r-1}$ (other case works analogously)}\label{existence.fig}
\end{center}
\end{figure}

 Now we have to show that $M_{pen}(\bm \beta_{f}) < M_{pen}(\bm \beta_{nf})$. Note that it depends on the design and the tuning etc. how small $\epsilon$ has to be such that the penalty decreases. We choose $\epsilon$ small enough which will be specified later. It holds
 \begin{eqnarray}
 	\bm \beta_{nf}-\bm \beta_{f}=(0,...,0,\epsilon_{1},-\epsilon_{2},0,...,0).
 \end{eqnarray}
Because of the continuity of the negative log likelihood $-L_{n}(\bm \beta)$ and the norm $||\bm \beta||_{\bm K}$ in every component, we can find $\delta_{1},\delta_{2}$ (which both depend on $\epsilon=\epsilon_{1}+\epsilon_{2}$) such that 
\begin{eqnarray*}
|L_{n}(\bm \beta_{f})	-L_{n}(\bm \beta_{\bm \beta_{nf}})| &< {\delta_{1}} \\
|||\bm \beta_{nf}||_{\bm K}-||\bm \beta_{f}||_{\bm K}| &< \delta_{2}
\end{eqnarray*}
Because of the definition of $\bm \beta_{nf}$ (no categories fused) and $\bm \beta_{f}$ (category $r$ and $r-1$ fused) we know that $\sum_{i=1}^{p}w_{0}^{(i)}||\beta_{nf,i}-\beta_{nf,i-1}||_{0}=\sum_{i}w_{0}^{(i)} =: c$ and for the fused version we know $\sum_{i=1}^{p}w_{0}^{(i)}||\beta_{f,i}-\beta_{f,i-1}||_{0}=c- w_{0}^{(r)}$. Furthermore $L_{n}(\bm \beta_{f})	-L_{n}(\bm \beta_{\bm \beta_{nf}})>-\delta_{1}$ and $||\bm \beta_{nf}||_{\bm K}-||\bm \beta_{f}||_{\bm K} > -\delta_{2}$.
Now we have 
\begin{eqnarray}
	M_{pen}(\bm \beta_{nf})-M_{pen}(\bm \beta_{f})&&=-L_{n}(\bm \beta_{nf})+\lambda_{1}||\bm \beta_{nf}||_{\bm K}+\lambda_{0}\sum_{i=1}^{p}w_{0}^{(i)}||\beta_{nf,i}-\beta_{nf,i-1}||_{0} \nonumber \\&&+L_{n}(\bm \beta_{f})-\lambda_{1}||\bm \beta_{f}||_{\bm K}-\lambda_{0}\sum_{i=1}^{p}w_{0}^{(i)}||\beta_{f,i}-\beta_{f,i-1}||_{0} \nonumber \\
	&&=-L_{n}(\bm \beta_{nf})+\lambda_{1}||\bm \beta_{nf}||_{\bm K}+L_{n}(\bm \beta_{f})-\lambda_{1}||\bm \beta_{f}||_{\bm K}+\lambda_{0}\cdot w_{0}^{(r)} \nonumber \\
	&& > - \delta_{1}-\lambda_{1}\delta_{2}+\lambda_{0}\cdot w_{0}^{(r)} \label{ex1}
\end{eqnarray}
Now, if we choose $\lambda_{0}$ (tuning for fusion) large enough and $\epsilon$ small enough (choice of $\epsilon$ affects $\delta_{1}$ and $\delta_{2}$) such that $\lambda_{0}\cdot w_{0}^{(r)} > \delta_{1}+\lambda_{1}\delta_{2}$, we get with the above equation 
\begin{eqnarray*}
M_{pen}(\bm \beta_{nf})- M_{pen}(\bm \beta_{f}) > 0 \Leftrightarrow	M_{pen}(\bm \beta_{nf}) > M_{pen}(\bm \beta_{f})
\end{eqnarray*}
and consequently the value of the objective function in $\bm \beta_{f}$ is less than in $\bm \beta_{nf}$, hence the objective function decreases if we fuse coefficients that are close enough to each other. The proof can directly be extended to the case where we fuse more categories and also for the nominal case. It makes sense that we have to choose $\lambda_{0}$ and $\epsilon$ in a specific way to enforce fusion because $\lambda_{0}$ is the tuning parameter for fusion and $\epsilon$ determines how close the two categories (or in particular their coefficients) are. 
\subsection{Proof of Theorem \ref{root-n-consistency}}
In the proof of Theorem \ref{root-n-consistency}, we will use the following Lemma.
\begin{lemma}\label{lemma.ball}
Let $M_{pen}(\bm \beta)$ be the objective function of $L_{0}$-FGL see (\ref{argminL0fgl}). Assume that we can show for some $\bm x^* \in \mathbb{R}^p$ and $c \in \mathbb{R}$ that
	\begin{eqnarray}
		\inf_{||\bm u||_{2}=c}M_{pen}(\bm x^*+\bm u) > M_{pen}(\bm x^*).\label{ass.lemma}
	\end{eqnarray}
	Then, there exists at least one local minimum of $M_{pen}(\bm \beta)$ inside $\mathcal{D}:=\{\bm x^*+\bm u \, |\, ||\bm u ||_{2} \leq c\}$, where inside means in the domain $\mathring{\mathcal{D}}=\{\bm x^*+\bm u \, |\, ||\bm u ||_{2} < c\}$.
\begin{proof}\hspace*{0cm}\\
\textit{Initial Remark:} If the function $M_{pen}$ was continuous, this would be clear since a continuous function attains its minimum and maximum in a compact set, hence in $\mathcal{D}$, and then we could use (\ref{ass.lemma}) to show that the infimum (minimum) is not attained at the boundary of $\mathcal{D}$. But, since $M_{pen}$ consists of an $L_{0}$ part, it is not continuous. Since we do not penalize the intercept and the intercept just appears in the log likelihood part, we neglect it hence we observe $M_{pen}(\bm \beta)$ for $\bm \beta \in \mathbb{R}^p$ and $\bm x^* \in \mathbb{R}^p$ (instead of $\mathbb{R}^{p+1})$. Consequently, we have to show that $M_{pen}$ attains its infimum in $\mathcal{D}$. Having that, we use (\ref{ass.lemma}) to show that the infimum is not attained at the boundary, hence it is in $\mathring{\mathcal{D}}$.\\

Returning to the proof, we will prove it for the case $p=2$ and $J=1$. Cases of higher dimensions work in a similar manner, although we get more possible cases for the infimum to occur (see below). In this setting of choosing $p=2$ and $J=1$, we just have one weight $w_{0}$ in the $L_{0}$ part (see (\ref{argminL0fgl}) again for the definition of the objective function). We start by partitioning $\mathcal{D}$ into two subsets in the following way 
\begin{eqnarray}
	\mathcal{D}_{1}&:=\mathcal{D}\backslash\{\bm \beta=(\beta_{1},\beta_{2})\,:\, \beta_{1}< \beta_{2}\},\\
		\mathcal{D}_{2}&:=\mathcal{D}\backslash\{\bm \beta=(\beta_{1},\beta_{2})\,:\, \beta_{2}< \beta_{1}\}.
\end{eqnarray}
So the hyperplane satisfying $\beta_{1}=\beta_{2}$ is included in both subsets. We clearly have that $\mathcal{D}=\mathcal{D}_{1}\cup \mathcal{D}_{2}$. This partition is displayed in figure \ref{lemma_ball_fig}.

\begin{figure}
\begin{center}
	\includegraphics[width=.5\textwidth]{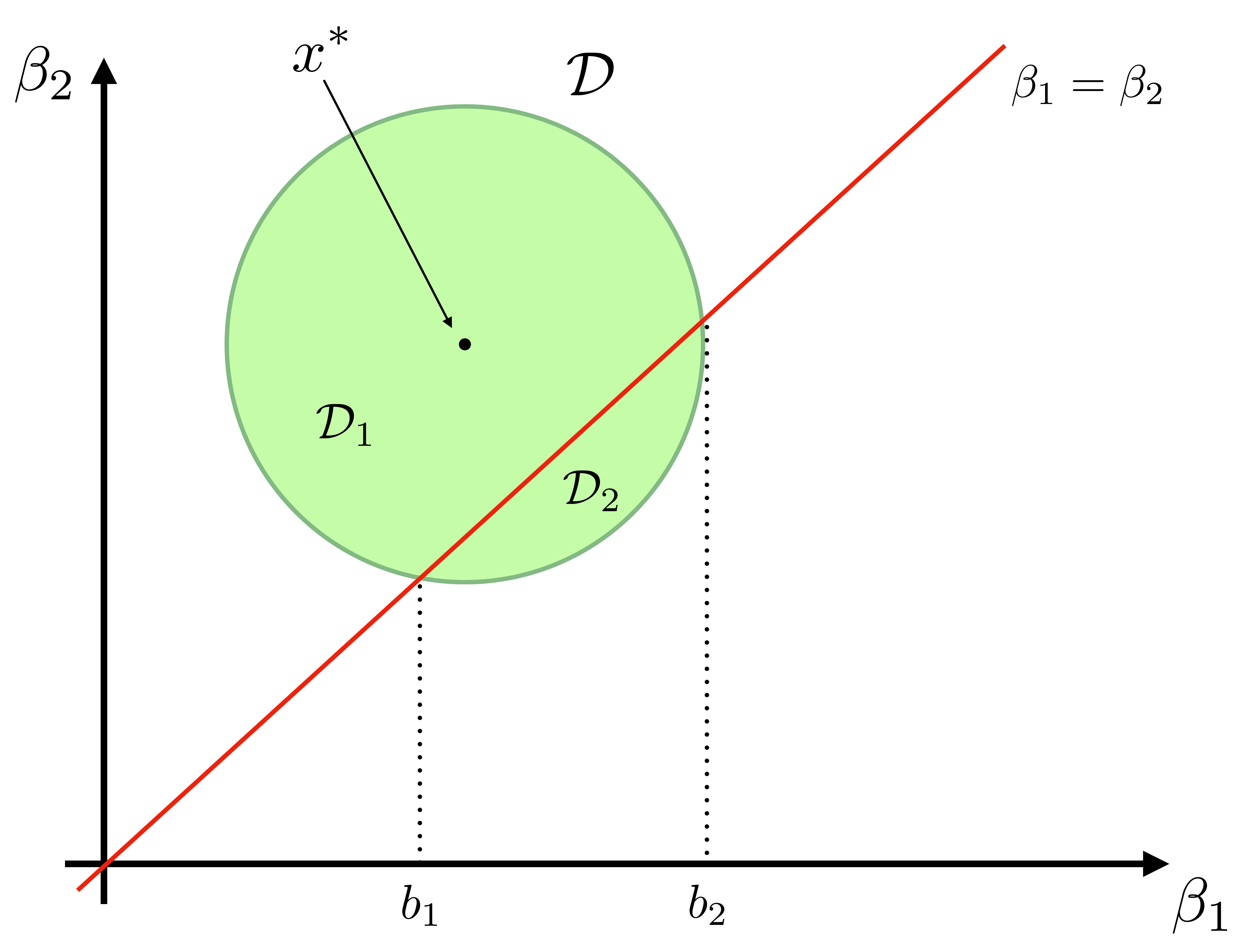}
	\caption{Partition of the ball $\mathcal{D}$ into $\mathcal{D}_{1}$ and $\mathcal{D}_{2}$. The red line shows the $1$-dimensional hyperplane where $f(\bm \beta)$ is not continuous, hence $\beta_{1}=\beta_{2}$}	
	\label{lemma_ball_fig}
\end{center}
\end{figure}
We can write by definition of the objective function $M_{pen}(\bm \beta)=g(\bm \beta)+f(\bm \beta)$ where $g(\bm \beta)$ is the sum of the log likelihood and group lasso part and $f(\bm \beta)$ the $L_{0}$ part. Note that, by definition of the $L_{0}$ norm applied to differences, this norm is equal to zero if the object on which we apply the norm is zero and one otherwise, hence it is zero if the difference is zero and it is one if the difference is nonzero. Keep in mind that we multiply the resulting value with the weight $w_{0}$.\\
 
 For $g(\bm \beta)$ we know that it attains a (local) minimum in $\mathcal{D}$, we write $\bm \beta_{g}=(\beta_{g,1},\beta_{g,2})= \argmin_{\bm \beta \in \mathcal{D}} g(\bm \beta)$. Without loss of generality, we assume that $\bm \beta_{g} \in \mathcal{D}_{1}$, the other case works completely analogous. There are two possible cases that may occur.
\begin{itemize}
\item[Case (1):] $\beta_{g,1} \neq \beta_{g,2}$ \\
Here, we have that $f(\bm \beta_{g})=f((\beta_{g,1},\beta_{g,2}))=1\cdot w_{0}=w_{0}$. Consequently, the infimum of the objective function either occurs in $\mathcal{D}_{1}$ without the hyperplane ($\beta_{1}=\beta_{2}$) or it occurs on this hyperplane. In particular, this means
\begin{eqnarray*}
	\inf_{\bm \beta \in \mathcal{D}}M_{pen}(\bm \beta) \in \{g(\bm \beta_{g})+w_{0}, \inf_{b \in [b_{1}, b_{2}]}g((b,b))\}
\end{eqnarray*}
so the infimum of $M_{pen}$ is either attained in $\bm \beta_{g}$ or in $(b,b)$ for some $b \in [b_{1},b_{2}]$. Later, we will show that with our additional assumption (\ref{ass.lemma}), we know that the infimum is not at the boundary hence $b \in (b_{1},b_{2})$ but this is not important at this point since we just want to show that the infimum is attained somewhere in $\mathcal{D}$.
\item[Case (2):] $\beta_{g,1} = \beta_{g,2}$ \\
In this case we have that $f(\bm \beta_{g})=f((\beta_{g,1},\beta_{g,1}))=0$. Consequently 
\begin{eqnarray*}
	\inf_{\bm \beta \in \mathcal{D}}M_{pen}(\bm \beta) = \inf_{\bm \beta \in \mathcal{D}} g(\bm \beta)\end{eqnarray*}
and hence the infimum of $M_{pen}$ is attained in $\bm \beta_{g}$. 
\end{itemize}
In both cases, there exists some $\tilde{\bm \beta}$ for which the infimum is attained, hence 
\begin{eqnarray*}
\argmin_{\bm \beta \in \mathcal{D}}M_{pen}(\bm \beta)= \tilde{\bm \beta}.
\end{eqnarray*}
Note that, in figure \ref{lemma_ball_fig} it can of course also occur the case that the red hyperplane does not go through the domain $\mathcal{D}$ hence there is no intersection of the hyperplane and $\mathcal{D}$. If this is the case, we are finished since then the function $f$ will be equal to one everywhere, hence $M_{pen}$ would be continuous. 
It remains to show that $\tilde{\bm \beta} \in \mathring{\mathcal{D}}$. Assume that $\tilde{\bm \beta}$ is on the boundary of $\mathcal{D}$, hence $\tilde{\bm \beta} \in \mathcal{D}\backslash \mathring{\mathcal{D}}$. Consequently, it holds by definition of the infimum that 
\begin{eqnarray*}
	\inf_{||\bm u||_{2}=c}M_{pen}(\bm x^*+\bm u) = M_{pen}(\tilde{\bm \beta}) \leq M_{pen}(\bm \beta) \, \, \, \forall \bm \beta \in \mathcal{D} 
\end{eqnarray*}
and this also holds for $\bm \beta = \bm x^*$ which is a contradiction to the assumption (\ref{ass.lemma}). Therefore, it holds that $\tilde{\bm \beta} \in \mathring{\mathcal{D}}$, hence there exists a local minimum of $M_{pen}$ in $\mathring{\mathcal{D}}$ so inside $\mathcal{D}$.
	\end{proof}
	\end{lemma}
	
Before proceeding to the proof of Theorem \ref{root-n-consistency}, we make an initial remark on this theorem.
\begin{rem}[on Theorem \ref{root-n-consistency}]\hspace*{0cm}
\begin{enumerate}
	\item Here, in the fixed case, $J$ and $p_{j} \,(j=1,...,J)$ are fixed so $a_{n}^1$ and $a_{n}^0$ always exist. 
	\item In the following proof, we will refer to \cite{FanLi2001}, proof of Theorem 1, where this work is about properties of nonconcave penalty functions (e.g., the SCAD penalty). In \cite{FanLi2001} (Theorem 1), they show that $||\hat{\bm \beta}-\bm \beta^*||_{2}=O_{p}(\frac{1}{\sqrt{n}}+ a_{n})$, where $a_{n}$ is the maximum of the derivative of the nonconcave penalty function (for example SCAD) and the tuning. Additionally, $\hat{\bm \beta}$ is the resulting estimate corresponding to the penalty function. They also argue in Remark 1, that for $\lambda_{n}\rightarrow 0$, one gets $a_{n}=0$ for SCAD. Here, in our setting, the $L_{0}$ part of our penalty function is not differentiable, hence we show $||\hat{\bm \beta}-\bm \beta^*||_{2}=O_{p}\left (\frac{1}{\sqrt{n}}\right )$ for $\hat{\bm \beta}$ being the $L_{0}$-FGL estimator.
	\item Note that the assumption for $a_{n}^{0}$ is stronger than for $a_{n}^{1}$, but for example \cite{FanLi2001} assumed for nonconvex penalties (SCAD) that $\lambda_{n} \rightarrow 0$ so our assumption is not too strong for $a_{n}^{0}$.
\end{enumerate}
\end{rem}

\begin{proof} (of Theorem \ref{root-n-consistency})
The $L_{0}$-FGL penalty function is known to be given by (see (\ref{argmin2}))
\begin{eqnarray*}
P_{\lambda}(\bm \beta)= \lambda_{n}^1\sum_{j=1}^J w_{1}^{(j)} ||\bm \beta_{j}||_{2} + \lambda_{n}^0 \sum_{j=1}^J \sum_{0 \leq r < s \leq p_{j}} w_{0}^{(j,rs)} ||\beta_{j,r}-\beta_{j,s}||_{0} 	
\end{eqnarray*}
and $M_{pen}(\bm \beta)=-L_{n}(\bm \beta)+P_{\lambda}(\bm \beta)$. Following \cite{FanLi2001} we have to show that $\forall \varepsilon > 0$ we can find a suitable $c>0$ such that the following holds
\begin{eqnarray}
P\left( \inf_{\bm u \in \mathbb{R}^p, ||\bm u||_{2}= c}M_{pen}\left(\bm \beta^* +\frac{1}{\sqrt{n}} \bm u\right) >  M_{pen}(\bm \beta^*)\right) \geq   1- \varepsilon.	\label{rootn}
\end{eqnarray}
In contrast to \cite{FanLi2001}, we minimize the sum of the negative log-likelihood and the chosen penalty where they maximize the negative objective function which clearly is the same. But, this is the reason why we have to show the opposite inequality in (\ref{rootn}) and we have to use the infimum instead of the supremum.
We will transfer the idea of \cite{FanLi2001} to our case of $L_{0}$-FGL including two penalties (GL and $L_{0}$) where the penalty function is not differentiable at any point and, another difference to the previously mentioned approach is that we have to handle with two types of tuning parameters and weights. Having shown (\ref{rootn}), we get that there exists a local minimum inside the ball $\{\bm \beta^* + \frac{1}{\sqrt{n}} \bm u\,\,\text{where}\, ||\bm u||_{2} < c\}$ using Lemma \ref{lemma.ball}. This yields that we can find a local minimizer such that $||\bm \beta^* -\hat{\bm \beta}||_{2}=O_{P}(1/ \sqrt{n})$ which is the claim.\\
We start by plugging in the definition of $M_{pen}(\cdot)$ giving us
\begin{eqnarray*}
&& M_{pen}\left(\bm \beta^* +\frac{1}{\sqrt{n}}\right) - M_{pen}(\bm \beta^*)\\
&=&-L_{n}\left(\bm \beta^*+\frac{1}{\sqrt{n}}\bm u\right)+L_{n}(\bm \beta^*) \lambda_{n}^1 \sum_{j=1}^J w_{1}^{(j)}\left(||\bm \beta_{j}^*+\frac{1}{\sqrt{n}}\bm u||_{2}-||\bm \beta_{j}^*||_{2} \right )\nonumber \\ && + \lambda_{n}^0 \sum_{j=1}^J \sum_{0 \leq r < s \leq p_{j}} w_{0}^{(j,rs)}\left(||\beta_{j,r}^*-\beta_{j,s}^*+ \frac{1}{\sqrt{n}}(u_{r}-u_{s})||_{0} - ||\beta_{j,r}^*-\beta_{j,s}^*||_{0}\right).
\end{eqnarray*}
We will observe the three parts from the right hand side of the equation seperately. Like in \cite{Zou2006} (proof of Theorem 4) we will investigate the behavior of the first part $-L_{n}\left(\bm \beta^*+\frac{1}{\sqrt{n}}\bm u\right )+L_{n}(\bm \beta^*)$ with a Taylor expansion of $f(\bm u):= -L_{n}\left(\bm \beta^*+\frac{1}{\sqrt{n}}\bm u\right )+L_{n}(\bm \beta^*)$ around $\bm u =0$ which gives us using $f(\bm 0)=\bm 0$
\begin{eqnarray}
	-L_{n}\left (\bm \beta^*+\frac{1}{\sqrt{n}}\bm u \right )+L_{n}(\bm \beta^*)=T_{1}^{(n)}+T_{2}^{(n)}+T_{3}^{(n)}\label{sum1}.
\end{eqnarray}
In particular, we have using (Reg1) and writing $\alpha_{n}=\frac{1}{\sqrt{n}}$
\begin{eqnarray*}
	&& T_{1}^{(n)}=-\alpha_{n}\nabla^TL_{n}(\bm \beta^*)\bm u=-\sum_{i=1}^n [y_{i}-\varphi^{'}(\bm x_{i}^T \bm \beta^*)]\bm x_{i}^T \bm u \,\alpha_{n} \\
&& T_{2}^{(n)}= -\frac{1}{2}\bm u^T \nabla^2 L_{n}(\bm \beta^*)\bm u \alpha_{n}^2=\sum_{i=1}^n \frac{1}{2} \varphi''(\bm x_{i}^T \bm \beta^*) \bm u^T \bm x_{i}\bm x_{i}^T\bm u \,\alpha_{n}^2\\
&& T_{3}^{(n)}= -\frac{1}{6}\sum_{i,j,k=1}^{p}\frac{\partial L_{n}(\bm \beta^*)}{\partial \beta_{i}\partial \beta_{j}\partial \beta_{k}}u_{i}u_{j}u_{k}\alpha_{n}^3=\alpha_{n}^3 \sum_{i=1}^n \frac{1}{6} \varphi'''(\bm x_{i}^T \bm \beta^*)(\bm x_{i}^T \bm u)^3.
\end{eqnarray*}
The last equality for $T_{3}^{(n)}$ can be directly seen if we plug in the required form of the log-likelihood. Using that $\alpha_{n}=\frac{1}{\sqrt{n}}$, we obtain the following asymptotic limits
\begin{eqnarray}
&& T_{1}^{(n)}=-\sum_{i=1}^n [y_{i}-\varphi^{'}(\bm x_{i}^T \bm \beta^*)]\frac{\bm x_{i}^T \bm u}{\sqrt{n}} \rightarrow_{d}N(\bm 0, \bm u^T I_{F}(\bm \beta^*)\bm u)	\,\,\, \text{(using CLT)}, \nonumber \\
&& T_{2}^{(n)}= \sum_{i=1}^n \frac{1}{2} \varphi''(\bm x_{i}^T \bm \beta^*) \bm u^T \frac{\bm x_{i}\bm x_{i}^T}{n} \bm u \rightarrow_{p} \frac{1}{2} \bm u^T I_{F}(\bm \beta^*) \bm u \,\,\, \text{(using LLN)},,\label{gw1}\\
&& T_{3}^{(n)}= n^{-1/2}\frac{1}{6} \underbrace {\frac{1}{n}\sum_{i=1}^n  \varphi'''(\bm x_{i}^T \bm \beta^*)(\bm x_{i}^T \bm u)^3}_{\rightarrow_{p} \mathbb{E}(M(\bm x)|\bm x^T \bm u|^3)\, <\, \infty\,\, \text{using (Reg3)}}\,\,\, \text{(using LLN)},\,\, \text{thus } \,\,6 \sqrt{n} T_{3}^{(n)} < \infty, \nonumber 
\end{eqnarray}
see also \cite{Zou2006} proof of Theorem 4. Of course the regularity conditions (Reg1)-(Reg3) are needed to get the asymptotic behavior above.  \\
With these properties we can conclude that the likelihood part of the objective function, hence (\ref{sum1}), is asymptotically dominated by ($\ref{gw1}$), thus by the expression $\bm u^T I_{F}(\bm \beta^*) \bm u$.\\
Since $||\bm \beta_{j}^*-\frac{1}{\sqrt{n}}\bm u||_{2} \leq ||\bm \beta_{j}^*||_{2}+||\frac{1}{\sqrt{n}}\bm u||_{2}$ we obtain $||\bm \beta_{j}^*-\frac{1}{\sqrt{n}}\bm u||_{2} -||\bm \beta_{j}^*||_{2}\leq ||\frac{1}{\sqrt{n}}\bm u||_{2}$. Consequently we admit 
\begin{eqnarray}
\lambda_{n}^1 \sum_{j=1}^J w_{1}^{(j)}\left(||\bm \beta_{j}^*+\frac{1}{\sqrt{n}}\bm u||_{2}-||\bm \beta_{j}^*||_{2} \right )\leq a_{n}^1 \frac{1}{\sqrt{n}}||\bm u||_{2}J\,, \, \, \text{so clearly also }\geq -a_{n}^1 \frac{1}{\sqrt{n}}||\bm u||_{2}J.\nonumber
\end{eqnarray}
Therefore this part of the sum given by the left hand side of the equation above is $O_{p}(1)||\bm u||.$
The last part to analyze is the following, where we use that for $L_{0}$ "norm" it holds $||...||_{0}\leq 1$ and consequently $||\beta_{j,r}^*-\beta_{j,s}^*+ \frac{1}{\sqrt{n}}(u_{r}-u_{s})||_{0} - ||\beta_{j,r}^*-\beta_{j,s}^*||_{0} \leq 1$, hence
\begin{eqnarray}
	&&\lambda_{n}^0 \sum_{j=1}^J \sum_{0 \leq r < s \leq p_{j}} w_{0}^{(j,rs)}\left(||\beta_{j,r}^*-\beta_{j,s}^*+ \frac{1}{\sqrt{n}}(u_{r}-u_{s})||_{0} - ||\beta_{j,r}^*-\beta_{j,s}^*||_{0}\right) \label{sum3}\\
	&&\leq \sum_{j=1}^J \sum_{0 \leq r < s \leq p_{j}} \,\,\,\underbrace{w_{0}^{(j,rs)} \lambda_{n}^0}_{\leq a_{n}^0\,\rightarrow \, K} \, \label{highdim1} 
\end{eqnarray}
giving us that (\ref{sum3}) is also $O_{P}(1)$. Note that $p$ and $p_{j}$ are fixed in this theorem thus they do not grow with the sample size $n$.\\
All in all, we can write 
\begin{eqnarray}
	&&M_{pen}\left(\bm \beta^* +\frac{1}{\sqrt{n}}\right) - M_{pen}(\bm \beta)\nonumber \\
	&=&T_{1}^{(n)}+T_{2}^{(n)}+T_{3}^{(n)}+\lambda_{n}^1 \sum_{j=1}^J w_{1}^{(j)}\left(||\bm \beta_{j}^*+\frac{1}{\sqrt{n}}\bm u||_{2}-||\bm \beta_{j}^*||_{2} \right ) \nonumber \\ 
	&&+ \lambda_{n}^0 \sum_{j=1}^J \sum_{0 \leq r < s \leq p_{j}} w_{0}^{(j,rs)}\left(||\beta_{j,r}^*-\beta_{j,s}^*+ \frac{1}{\sqrt{n}}(u_{r}-u_{s})||_{0} - ||\beta_{j,r}^*-\beta_{j,s}^*||_{0}\right) \nonumber \\
	&=& \underbrace{T_{1}^{(n)}}_{\rightarrow N(...)}+\underbrace{T_{2}^{(n)}}_{\rightarrow \frac{1}{2} \bm u^T I_{F}(\bm \beta^*) \bm u}+\underbrace{T_{3}^{(n)}}_{\text{bounded}} +O_{p}(1)||\bm u||+ O_{p}(1) \label{last}
\end{eqnarray}
We conclude that the expression $M_{pen}\left(\bm \beta^* +\frac{1}{\sqrt{n}}\right) - M_{pen}(\bm \beta)$ is dominated (asymptotically) by $\frac{1}{2}\bm u^T I_{F}(\bm \beta^*) \bm u > 0$ where this expression is positive since the Fisher Information matrix was assumed to be positive definite at $\bm \beta^*$. Hence, for $n$ large enough, we can choose $c$ in such a way (in particular it has to be large enough) that (\ref{last}) $>0$ hence (\ref{rootn}) holds so there exists $\hat{\bm \beta}$ being $\sqrt{n}$-consistent.
\end{proof}	

\subsection{Proof of Theorem \ref{consistency-highdim}}
\begin{rem}[Initial Remark on Theorem \ref{consistency-highdim}]\hspace*{0cm}
\begin{enumerate}
	\item Note that as $J=J_{n}$ depends on $n$, the quantity $p_{n}=\sum_{j=1}^{J_{n}} p_{j}$ also depends on $n$.
	\item One could also additionally assume that the number of levels is bounded, hence $\max\{p_{j} | j=1,...,J_{n}\} = c_{3} < \infty$ for some constant $c_{3} >0$. In this case, the assumption $a_{n}^0J_{n}p_{n}(p_{n}-1)\rightarrow c_{2}$ would simplify to $a_{n}^0J_{n}c_{3}(c_{3}-1)\rightarrow c_{2}$ as $n \rightarrow \infty$, since the number of possible differences for each factor is bounded by  $\frac{c_{3}(c_{3}-1)}{2}$ which will result in  $a_{n}^0J_{n}\rightarrow c_{2}/c_{3}(c_{3}-1)$ as we will see when we execute the proof.
\end{enumerate}
\end{rem}

\begin{proof}
The proof is related to the proof of Theorem 1 in \cite{FanPeng2004}, where such a theorem is shown for nonconcave penalties as SCAD. We will transfer the idea to our case of $L_{0}$-FGL including two penalties (group lasso and $L_{0}$) where the penalty function is not differentiable at any point and, another difference to the previously mentioned approach is that we have to handle with two types of tuning parameters and weights. In \cite{FanPeng2004} (Theorem 1) the weights are chosen to be equal to one. The first part of the proof, where the log-likelihood is observed, is similar to \cite{FanPeng2004}. As in the proof of Theorem \ref{root-n-consistency}, we show that for any given $\varepsilon>0$, we can find a suitable $c$ (large enough) to ensure that
\begin{eqnarray*}
	P\left(\inf_{\bm u \in \mathbb{R}^{p_{n}}, ||\bm u||_{2}=c}M_{pen}(\bm \beta^* +\alpha_{n}\bm u) >  M_{pen}(\bm \beta^*)\right) \geq 1-\varepsilon,
\end{eqnarray*}
where $\bm \beta^*$ is the true underlying parameter vector and $\alpha_{n}=\sqrt{\frac{p_{n}}{n}}$. If we showed the above inequality, we could conclude that we get a probability of at least $1-\varepsilon$ that there exists a local minimum of the objective function inside of the ball $\{\bm \beta^*+\alpha_{n}\bm u \,:\, ||\bm u||_{2}\leq c\}$ using Lemma \ref{lemma.ball}. Consequently there exists a local minimizer $\hat{\bm \beta}$ of the objective function satisfying $||\hat{\bm \beta}-\bm \beta^*||_{2}=O_{p}(\alpha_{n})$. As in the proof of Theorem \ref{asynorm}, we define $H_{n}(\bm u):=M_{pen}(\bm \beta^*+\alpha_{n}\bm u)-M_{pen}(\bm \beta^*)$, where in the proof of Theorem \ref{asynorm} $\alpha_{n}$ corresponds to $\frac{1}{\sqrt{n}}$, and obtain
\begin{eqnarray*}
	H_{n}(\bm u)=-L_{n}(\bm \beta^*+\alpha_{n}\bm u)+L_{n}(\bm \beta^*)+\lambda_{n}^1 \sum_{j=1}^{J_{n}}(w_{1}^{(j)}||\bm \beta_{j}^*+\alpha_{n}\bm u_{j}||_{2}-w_{1}^{(j)}||\bm \beta_{j}||_{2})\\
	+\lambda_{n}^0 \sum_{j=1}^{J_{n}}\sum_{0 \leq r < s \leq p_{j}}(w_{j}^{(j,rs)}||\beta_{j,r}^*-\beta_{j,s}^*+\alpha_{n}(u_{r}-u_{s})||_{0}-w_{0}^{(j,rs)}||\beta_{j,r}^*-\beta_{j,s}^*||_{0}).
\end{eqnarray*}
We will observe the lof likelihood part and the penalty part of the objective function seperately. \\

\textbf{Step 1: Log Liklelihood} \\
For the log likelihood part we perform a Taylor expansion as in the proofs of Theorems \ref{root-n-consistency} and \ref{asynorm} but since we are in the case that $p_{n}$ grows with $n$, the observation of the behavior of the components of the Taylor expansion will differ from the mentioned theorems. In particular, we get for the Taylor expansion of $f(\bm u):=-L_{n}(\bm \beta^*+\alpha_{n}\bm u)+L_{n}(\bm \beta^*)$ around $\bm u=\bm 0$ using the fact that $f(\bm 0)= 0$
\begin{eqnarray*}
	-L_{n}(\bm \beta^*+\alpha_{n}\bm u)+L_{n}(\bm \beta^*)=T_{1}^{(n)}+T_{2}^{(n)}+T_{3}^{(n)}.
\end{eqnarray*}
Please note that we showed $\sqrt{n}$ consistency in Theorem \ref{root-n-consistency} whereas here we show $\alpha_{n}^{-1}=\sqrt{\frac{n}{p_{n}}}$ consistency, hence the form of $T_{i}^{(n)}$ for $i \in \{1,2,3\}$ slightly differ (we multiply by $\alpha_{n}$ instead of $\frac{1}{\sqrt{n}}$). In particular, it holds similarly to the proof of Theorem \ref{root-n-consistency}
\begin{eqnarray*}
&& T_{1}^{(n)}=-\alpha_{n}\nabla^TL_{n}(\bm \beta^*)\bm u\\
&& T_{2}^{(n)}= -\frac{1}{2}\bm u^T \nabla^2 L_{n}(\bm \beta^*)\bm u \alpha_{n}^2\\
&& T_{3}^{(n)}= -\frac{1}{6}\sum_{i,j,k=1}^{p_{n}}\frac{\partial L_{n}(\bm \beta^*)}{\partial \beta_{i}\partial \beta_{j}\partial \beta_{k}}u_{i}u_{j}u_{k}\alpha_{n}^3.
\end{eqnarray*}
Please note that these expressions of $T_{i}^{(n)}, \, i \in \{1,2,4\}$ are the same as in Theorem \ref{root-n-consistency} if we plug in $f(\bm v, \bm \beta)$ which gives us a particular form of the log likelihood $L_{n}$ and hence we can simplify the expressions to the expressions given in Theorem \ref{root-n-consistency}. Nevertheless, as we focus on another way of analyzing the summands, we work with the expressions above. Keep in mind that the likelihood is given by $L_{n}(\bm \beta)=\sum_{i=1}^n \log(f_{n}(\bm v_{i}, \bm \beta))$.

For $T_{1}^{(n)}$ we get using the Cauchy-Schwartz inequality and (div.Reg2) 
\begin{eqnarray*}
	|T_{1}^{(n)}|&=&|\alpha_{n}\nabla^TL_{n}(\bm \beta^*)\bm u| \\
	&\leq & \alpha_{n}||\nabla^T L_{n}(\bm \beta^*)||_{2}\, ||\bm u||_{2} \\
	&=&O_{p}(\alpha_{n}\sqrt{np_{n}})||\bm u||_{2}
	=O_{p}(\alpha_{n}^2 n) ||\bm u||_{2}=O_{p}(p_{n}) ||\bm u||_{2}
	\end{eqnarray*}
	since
	\begin{eqnarray*}
		||\nabla^T L_{n}(\bm \beta^*)||_{2}^2&&=\sum_{j=1}^{p_{n}} \frac{\partial L_{n}(\bm \beta^*)}{\partial \beta_{j}}\frac{\partial L_{n}(\bm \beta^*)}{\partial \beta_{j}}\\
		&& =n \sum_{j=1}^{p_{n}} \underbrace{\frac{1}{n}\sum_{i=1}^n \frac{\partial \log f_{n}(v_{i},\bm \beta^*)}{\partial \beta_{j}}\frac{\partial \log f_{n}(v_{i},\bm \beta^*)}{\partial \beta_{j}}}_{\rightarrow_{p}\,\, \mathbb{E}\left (\frac{\partial \log f_{n}(v_{i},\bm \beta^*)}{\partial \beta_{j}}\frac{\partial \log f_{n}(v_{i},\bm \beta^*)}{\partial \beta_{j}}\right )= [\bm I_{F}(\bm \beta^*)]_{j,j}\, <\, C_{4}} \\
		&& = p_{n}\,n\,O_{p}(1) \\
	\Rightarrow ||\nabla^T L_{n}(\bm \beta^*)||_{2} &&= O_{p}(\sqrt{n\,p_{n}}).
	\end{eqnarray*}
	So in particular we can write $T_{1}^{(n)}=O_{p}(p_{n}) ||\bm u||_{2}$ since $\alpha_{n}^{2}n=p_{n}=\alpha_{n}\sqrt{np_{n}}$.\\
	For the second summand $T_{2}^{(n)}$ it holds as in \cite{FanPeng2004}
	\begin{eqnarray*}
		T_{2}^{(n)}&&=-\frac{1}{2}\bm u^T \nabla^2 L_{n}(\bm \beta^*)\bm u \alpha_{n}^2 \\
		&&= -\frac{1}{2}\bm u^T \nabla^2 L_{n}(\bm \beta^*)\bm u \alpha_{n}^2 \underbrace{+\frac{1}{2}\bm u^T \bm I_{F}(\bm \beta^*)\bm u n \alpha_{n}^2 -\frac{1}{2}\bm u^T \bm I_{F}(\bm \beta^*)\bm u n \alpha_{n}^2}_{=0}\\
				&&= -\frac{1}{2}\bm u^T\left[\frac{1}{n}(\nabla^2 L_{n}(\bm \beta^*)+\bm I_{F}(\bm \beta^*))\right]\bm u n \alpha_{n}^2+\frac{1}{2}\bm u^T \bm I_{F}(\bm \beta^*) \bm u n \alpha_{n}^2 \\
			&&= -\frac{1}{2}\bm u^T\left[\frac{1}{n}(\nabla^2 L_{n}(\bm \beta^*)-E(\nabla^2L_{n}(\bm \beta^*))\right]\bm u n \alpha_{n}^2+\frac{1}{2}\bm u^T \bm I_{F}(\bm \beta^*) \bm u n \alpha_{n}^2 \\
		&&=\frac{n}{2}\alpha_{n}^2\bm u^T o_{p}(1/p_{n})\bm u+\frac{n}{2}\alpha_{n}^2 \bm u^T \bm I_{F}(\bm \beta^*)\bm u
	\end{eqnarray*}
	where we used that $\bm I_{F}(\bm \beta^*)=-E(\nabla^2 L_{n}(\bm \beta^*))$ and $||\frac{1}{n}\nabla^2 L_{n}(\bm \beta^*)+\bm I_{F}(\bm \beta^*)|| =o_{p}(\frac{1}{p_{n}})$ following Lemma 8 of \cite{FanPeng2004} which needs the assumption $p_{n}^4/n \rightarrow 0$ as $n\rightarrow \infty$. Since we have  $||\frac{1}{n}\nabla^2 L_{n}(\bm \beta^*)+\bm I_{F}(\bm \beta^*)|| =o_{p}(\frac{1}{p_{n}})$ we know that, by definition, $||\frac{1}{n}\nabla^2 L_{n}(\bm \beta^*)+\bm I_{F}(\bm \beta^*)||\,p_{n}$ converges to zero in probability so using $p_{n}\geq 1$ we get $||\frac{1}{n}\nabla^2 L_{n}(\bm \beta^*)+\bm I_{F}(\bm \beta^*)||\,p_{n}\geq ||\frac{1}{n}\nabla^2 L_{n}(\bm \beta^*)+\bm I_{F}(\bm \beta^*)||$ hence the r.h.s. converges also to zero in probability so the r.h.s. is $o_{p}(1)$. Consequently, 
	\begin{eqnarray*}
		T_{2}^{(n)}&=&\frac{n}{2}\alpha_{n}^2 \bm u^T o_{p}(1)\bm u+\frac{n}{2}\alpha_{n}^2 \bm u^T \bm I_{F}(\bm \beta^*)\bm u\\
		&=& \frac{1}{2}p_{n}\bm u^T(\bm I_{F}(\bm \beta^*)+o_{p}(1))\bm u.
	\end{eqnarray*}
	The last summand $T_{3}^{(n)}$ is treated as follows. 
	\begin{eqnarray*}
		|T_{3}^{(n)}|&&=\frac{1}{6}\left|\sum_{i,j,k=1}^{p_{n}}\frac{\partial^3 L_{n}(\bm \beta^*)}{\partial \beta_{i}\partial \beta_{j}\partial \beta_{k}}u_{i}u_{j}u_{k}\alpha_{n}^3\right|= \frac{1}{6}\left| \sum_{l=1}^n \sum_{i,j,k}^{p_{n}}\frac{\partial^3\log f_{n}(\bm v_{l},\bm \beta^*)}{\partial \beta_{i}\partial \beta_{j}\partial \beta_{k}}u_{i}u_{j}u_{k}\alpha_{n}^3 \right| \\
		&& \leq \frac{1}{6}\alpha_{n}^3 \underbrace{ \sum_{l=1}^n \left|\sum_{i,j,k}^{p_{n}}\frac{\partial^3\log f_{n}(\bm v_{l},\bm \beta^*)}{\partial \beta_{i}\partial \beta_{j}\partial \beta_{k}}u_{i}u_{j}u_{k}\right|}_{(*)},
		\end{eqnarray*}
		where, using the Cauchy Schwartz inequality we obtain
		\begin{eqnarray*}
			(*)=\sum_{l=1}^n \left|\sum_{i,j,k}^{p_{n}}\frac{\partial^3\log f_{n}(\bm v_{l},\bm \beta^*)}{\partial \beta_{i}\partial \beta_{j}\partial \beta_{k}}u_{i}u_{j}u_{k}\right|  \leq  \sum_{l=1}^n || (\log f_{n}(\bm v_{l},\bm \beta^*))'''||_{2} \cdot|| \bm u||_{2} ^3,
		\end{eqnarray*}
		where $||\cdot ||_{2}$ is the euclidean norm. Following (div.Reg3), we know that we can bound every component in $|| (\log f_{n}(\bm v_{l},\bm \beta^*))'''||_{2}$ by some function $M_{n,i,j,k}(\bm x_{l})$, hence
		\begin{eqnarray*}
			 \sum_{l=1}^n || (\log f_{n}(\bm v_{l},\bm \beta^*))'''||_{2} \leq \sum_{l=1}^n \left(\sum_{i,j,k=1}^{p_{n}} M_{n,i,j,k}^2(\bm x_{l})\right)^{1/2}
		\end{eqnarray*}
		so consequently
		\begin{eqnarray}
			(*) \leq \sum_{l=1}^n || (\log f_{n}(\bm v_{l},\bm \beta^*))'''||_{2} \cdot|| \bm u||_{2} ^3 \leq ||\bm u||_{2}^3 \sum_{l=1}^n \left(\sum_{i,j,k=1}^{p_{n}} M_{n,i,j,k}^2(\bm x_{l})\right)^{1/2}. \label{ineq1}
		\end{eqnarray}
	Now we have to observe the asymptotic behavior of the r.h.s. of the inequality $(\ref{ineq1})$. Using the Cauchy Schwarz inequality we obtain 
	\begin{eqnarray}
		&& \left(\sum_{i,j,k=1}^{p_{n}} M_{n,i,j,k}^2(\bm x_{l})\right)^{2} = \left(\sum_{i,j,k=1}^{p_{n}} M_{n,i,j,k}^2(\bm x_{l})\cdot 1 \right)^{2} \leq \left(\sum_{i,j,k=1}^{p_{n}} M_{n,i,j,k}^2(\bm x_{l})\right) p_{n}^3 \nonumber \\
		&& \Rightarrow \sum_{i,j,k=1}^{p_{n}} M_{n,i,j,k}^2(\bm x_{l}) \leq p_{n}^3 \nonumber \\
		&& \Rightarrow \left(\sum_{i,j,k=1}^{p_{n}} M_{n,i,j,k}^2(\bm x_{l})\right)^{1/2} \leq p_{n}^{3/2}\label{cauchyschw1}
	\end{eqnarray}
	With (\ref{cauchyschw1}) we can write using $\alpha_{n}=\sqrt{p_{n}/n}$
	\begin{eqnarray*}
		|T_{3}^{(n)}| &\leq & \frac{1}{6}\alpha_{n}^3 ||\bm u||_{2}^3 \,\sum_{l=1}^n \left(\sum_{i,j,k=1}^{p_{n}} M_{n,i,j,k}^2(\bm x_{l})\right)^{1/2} \\
		&\leq & \frac{1}{6}\alpha_{n}^3 ||\bm u||_{2}^3 \,\sum_{l=1}^n p_{n}^{3/2} =\frac{1}{6}\alpha_{n}^3 ||\bm u||_{2}^3 n p_{n}^{3/2}= \frac{1}{6}||\bm u||_{2}^3 \frac{p_{n}^3}{\sqrt{n}}		\end{eqnarray*}
		Since we assumed $\frac{p_{n}^4}{n}\rightarrow 0$ we get using $0 \leq \frac{p_{n}^2}{\sqrt{n}} = \sqrt{\frac{p_{n}^4}{n}} \leq \frac{p_{n}^4}{n}\rightarrow 0$ that $\frac{p_{n}^2}{\sqrt{n}} \rightarrow 0$. Consequently, it holds that $\frac{p_{n}^3}{\sqrt{n}}=o_{p}(p_{n})$.
		Hence summing up $T_{3}^{(n)}=o_{p}(p_{n})||\bm u||_{2}^3$. \\
		\textbf{Step 2: Penalty}\\
		Now we can come to the penalty parts where we start with the GL part before we come to the $L_{0}$ part of the $L_{0}$-FGL penalty function. 
\begin{eqnarray*}
	|\lambda_{n}^1 \sum_{j=1}^{J_{n}}(w_{1}^{(j)}||\bm \beta_{j}^*+\alpha_{n}\bm u_{j}||_{2}-w_{1}^{(j)}||\bm \beta_{j}^*||_{2})| && \leq \lambda_{n}^1 \sum_{j=1}^{J_{n}} w_{1}^{(j)}\alpha_{n}||\bm u_{j}||_{2} \\
	&& \leq ||\bm u||_{2} \,\alpha_{n} \sum_{j=1}^{J_{n}} \lambda_{n}^1 w_{1}^{(j)}\\
	&& \leq ||\bm u||_{2} \, \alpha_{n} a_{n}^1 J_{n} \\
	&&=O_{p}(1)||\bm u||_{2}
	\end{eqnarray*}
	since by assumption $\alpha_{n}a_{n}^1 J_{n}\rightarrow c_{1}$ as $n \rightarrow \infty$.
	Lastly, since $||\beta_{j,r}^*-\beta_{j,s}^*+\alpha_{n}(u_{r}-u_{s})||_{0}-||\beta_{j,r}^*-\beta_{j,s}^*||_{0} \leq 1$, we obtain
	\begin{eqnarray*}
	&& \lambda_{n}^0\sum_{j=1}^{J_{n}}\sum_{0 \leq r < s \leq p_{j}}(w_{0}^{(j,rs)}||\beta_{j,r}^*-\beta_{j,s}^*+\alpha_{n}(u_{r}-u_{s})||_{0}-w_{0}^{(j,rs)}||\beta_{j,r}^*-\beta_{j,s}^*||_{0}) \\
	& \leq &  \sum_{j=1}^{J_{n}}\sum_{0 \leq r < s \leq p_{j}} \lambda_{n}^0 w_{j}^{(j,rs)}\\
	& \leq &  a_{n}^0 \sum_{j=1}^{J_{n}}\sum_{0 \leq r < s \leq p_{j}} 1 :=(*)
	\end{eqnarray*}
	The quantity $\sum_{j=1}^{J_{n}}\sum_{0 \leq r < s \leq p_{j}} 1$ is equal to the number of differences including all $J_{n}$ predictors of the model. Of course, this depends on the design whether we observe ordinal or nominal covariates, or mixtures. The highest number of possible differences occurs when all covariates are nominal, hence it can be bounded by $\frac{p_{n}(p_{n}-1)}{2}$ where we remember that $p_{n}$ is the total number of levels of all covariates. So it holds 
	\begin{eqnarray*}
		 \sum_{j=1}^{J_{n}}\sum_{0 \leq r < s \leq p_{j}} 1 \leq \frac{p_{n}(p_{n}-1)}{2}.
	\end{eqnarray*}
	Additionally, we assumed that $a_{n}^0 p_{n}(p_{n}-1) \rightarrow c_{2}$ as $n \rightarrow \infty$, hence we get $a_{n}^0 p_{n}(p_{n}-1) =O_{p}(1)$ and finally \\
	\begin{eqnarray*}
 (*)=a_{n}^0 \sum_{j=1}^{J_{n}}\sum_{0 \leq r < s \leq p_{j}} 1 \leq a_{n}^0 \frac{p_{n}(p_{n}-1)}{2}=O_{p}(1)
	\end{eqnarray*}
	so the $L_{0}$ part of the penalty function is $O_p(1)$.
	
	Now that we observed all parts seperately, we can conclude 
		\begin{eqnarray*}
	H_{n}(\bm u)&=&-L_{n}(\bm \beta^*+\alpha_{n}\bm u)+L_{n}(\bm \beta^*)+\lambda_{n}^1 \sum_{j=1}^{J_{n}}(w_{1}^{(j)}||\bm \beta_{j}^*+\alpha_{n}\bm u_{j}||_{2}-w_{1}^{(j)}||\bm \beta_{j}||_{2})\\
	&+&\lambda_{n}^0 \sum_{j=1}^{J_{n}}\sum_{0 \leq r < s \leq p_{j}}(w_{j}^{(j,rs)}||\beta_{j,r}^*-\beta_{j,s}^*+\alpha_{n}(u_{r}-u_{s})||_{0}-w_{0}^{(j,rs)}||\beta_{j,r}^*-\beta_{j,s}^*||_{0}) \\
	&=& \underbrace{T_{1}^{(n)}}_{=O_{p}(p_{n}) ||\bm u||}+\underbrace{T_{2}^{(n)}}_{=\frac{1}{2}p_{n}\bm u^T(\bm I_{F}(\bm \beta^*)+o_{p}(1))\bm u.}+\underbrace{T_{3}^{(n)}}_{=o_{p}(p_{n})||\bm u||^2}\\
	&+&\underbrace{\lambda_{n}^1 \sum_{j=1}^{J_{n}}(w_{1}^{(j)}||\bm \beta_{j}^*+\alpha_{n}\bm u_{j}||_{2}-w_{1}^{(j)}||\bm \beta_{j}||_{2})}_{= O_{p}(1)||\bm u||_{2}}\\
	&+&\underbrace{\lambda_{n}^0 \sum_{j=1}^{J_{n}}\sum_{0 \leq r < s \leq p_{j}}(w_{j}^{(j,rs)}||\beta_{j,r}^*-\beta_{j,s}^*+\alpha_{n}(u_{r}-u_{s})||_{0}-w_{0}^{(j,rs)}||\beta_{j,r}^*-\beta_{j,s}^*||_{0})}_{= O_{p}(1)}.
\end{eqnarray*}
We can see that all the summands are dominated by $\frac{1}{2}p_{n} \bm u^T\bm I_{F}(\bm \beta^*)\bm u >0 $ where the last inequality holds since we assumed that the Fisher information matrix is positive definite in $\bm \beta=\bm \beta^*$ in (div.Reg2). So choosing $c$ large enough, we can ensure that the whole r.h.s. of the equation above is $> 0$, hence $H_{n}(\bm u) > 0$. In particular, note that for a smaller value of $\varepsilon$, we have to choose a larger $c$. 
\end{proof}

\subsection{Proof of Theorem \ref{asynorm}}
\begin{proof}
 The proof follows \cite{Zou2006} where the oracle properties for the adaptive lasso are shown. We write $\bm \beta= \bm \beta^* + \frac{\bm u}{\sqrt{n}}$ and remember that $H_{n}(\bm u):=M_{pen}\left(\bm \beta^*+\frac{\bm u}{\sqrt{n}}\right )-M_{pen}(\bm \beta)$. We aim to minimize $\hat{\bm u}_{n}= \argmin_{\bm u} H_{n}(\bm u)$, then $\hat{\bm u}_{n}= \sqrt{n}(\hat{\bm \beta}_{n}-\bm \beta^*)$.
	It holds
	\begin{eqnarray}
	&& H_{n}(\bm u)\nonumber \\
	&&= \underbrace{\sum_{i=1}^{n} \left[-y_{i}(x_{i}^T(\bm \beta^* + \frac{\bm u}{\sqrt{n}}))+ \varphi(x_{i}^T(\bm \beta^* +\frac{\bm u}{\sqrt{n}}))\right]- \sum_{i=1}^n \left[-y_{i}(x_{i}^T \bm \beta^*)+\varphi(x_{i}^T \bm \beta^*)\right]}_{=-L_{n}(\bm \beta^*+\frac{\bm u}{\sqrt{n}})+L_{n}(\bm \beta^*)}	\nonumber \\
	&&\hspace*{0.3cm}- \lambda_{n}^{1} \sum_{j=1}^{J} \left[w_{1}^{(j)} ||\bm \beta_{j}^* +\frac{\bm u_{j}}{\sqrt{n}}||_{2}-w_{1}^{(j)}||\bm \beta_{j}^*||_{2}\right]\nonumber \\
	&& \hspace*{0.3cm}+ \lambda_{n}^{0} \sum_{j=1}^{J} \sum_{0 \leq r < s \leq p_{j}}\left[w_{0}^{(j,rs)}||\beta_{j,r}^*-\beta_{j,s}^*+\frac{u_{j,r}-u_{j,s}}{\sqrt{n}}||_{0}\right] \nonumber\\
	&& \hspace*{0.3cm}- \left[ w_{0}^{(j,rs)}||\beta_{j,r}^*-\beta_{j,s}^*||_{0} \right] \nonumber\\
	&& = \sum_{i=1}^n \left[-y_{i}x_{i}^T \frac{\bm u}{\sqrt{n}}+\varphi(x_{i}^T(\bm \beta^*+\frac{\bm u}{\sqrt{n}}))-\varphi(x_{i}^T \bm \beta^*)\right] \label{Taylor}\\
	&& \hspace*{0.3cm}- \lambda_{n}^1 \sum_{j=1}^J \left[w_{1}^{(j)} ||\bm \beta_{j}^* +\frac{\bm u_{j}}{\sqrt{n}}||_{2}-w_{1}^{(j)}||\bm \beta_{j}^*||_{2}\right] \label{a3}\\
	&& \hspace*{0.25cm}+ \lambda_{n}^{0} \sum_{j=1}^J \sum_{0 \leq r < s \leq p_{j}}\left[w_{0}^{(j,rs)}||\beta_{j,r}^*-\beta_{j,s}^*+\frac{u_{j,r}-u_{j,s}}{\sqrt{n}}||_{0}-w_{0}^{(j,rs)}||\beta_{j,r}^*-\beta_{j,s}^*||_{0} \right]\label{a5}
	\end{eqnarray}
Now, we will execute a taylor expansion of (\ref{Taylor}) as in (\ref{sum1}). Writing $f(\bm u):= (\ref{Taylor})$ and executing the Taylor expansion around $\bm u = \bm 0$ giving us $f(\bm u)= T_{1}^{(n)}+T_{2}^{(n)}+T_{3}^{(n)}$ using $f(\bm 0)=\bm 0$, see (\ref{sum1}). Now we will analzye the asymptotic behavior of the components of $(\ref{Taylor})$ and (\ref{a3}) as well as (\ref{a5}). Using the fact that the distribution of the response giving the data is member of an exponential family, we know that the following holds 
\begin{eqnarray*}
&& \mathbb{E}([y_{i}-\varphi^{'}(\bm x_{i}^T \bm \beta^*)](\bm x_{i}^T \bm u))=0, \\
&& \text{Var}([y_{i}-\varphi^{'}(\bm x_{i}^T \bm \beta^*)](\bm x_{i}^T \bm u))= \mathbb{E}(\varphi^{''}(\bm x_{i}^T \bm \beta^*)(\bm x_{i}^T \bm u)^2)= \bm u^T \bm I_{F}(\bm \beta^*) \bm u	.
\end{eqnarray*}
	Having that, the central limit theorem (CLT) gives us 
	\begin{eqnarray*}
	T_{1}^{(n)} \rightarrow_{d} \bm u^T N(\bm 0, \bm I_{F}(\bm \beta^*))	.
	\end{eqnarray*}
In addition 
\begin{eqnarray*}
\sum_{i=1}^n \varphi^{''}(\bm x_{i}^T \bm \beta^*) \frac{\bm x_{i} \bm x_{i}^T}{n} \rightarrow_{p} \bm I_{F}(\bm \beta^*),	
\end{eqnarray*}
thus
\begin{eqnarray*}
T_{2}^{(n)}	= \sum_{i=1}^n \frac{1}{2}\varphi^{''}(\bm x_{i}^T \bm \beta^*) \bm u^T \frac{\bm x_{i} \bm x_{i}^T}{n} \bm u \rightarrow_{p} \frac{1}{2}\bm u^T \bm I_{F}(\bm \beta^*)\bm u.
\end{eqnarray*}
For $T_{3}^{(n)}$ we can use the assumed regularity conditions to bound the expression (using LLN)
\begin{eqnarray*}
6 \sqrt{n} T_{3}^{(n)}	\leq \sum_{i=1}^n \frac{1}{n} M(\bm x_{i}^T)|\bm x_{i}^T \bm u|^3 \rightarrow_{p} \mathbb{E}(M(\bm x)|\bm x^T \bm u|^3) < \infty.
\end{eqnarray*}
Hence, it remains to analyze the asymptotic behavior of $P_{1}^{(n)}:=$ (\ref{a3}), hence the group lasso penalty party and $P_{0}^{(n)}:=$ (\ref{a5}), hence the $L_{0}$ penalty part. For
\begin{eqnarray*}
P_{1}^{(n)}= \lambda_{n}^1 \sum_{j=1}^J \left[ w_{1}^{(j)} ||\bm \beta_{j}^* +\frac{\bm u_{j}}{\sqrt{n}}||_{2} -w_{1}^{(j)}||\bm \beta_{j}^*||_{2}\right]	
\end{eqnarray*}
we get for the case that $\bm \beta_{j}^* \neq \bm 0$ the following (where $\tilde{\bm \beta}$ is the classical MLE) 
\begin{eqnarray*}
w_{1}^{(j)}&=& \frac{1}{||\tilde{ \bm \beta}_{j}||_{2}^{\gamma}} \rightarrow_{p} ||\bm \beta^*_{j}||_{2}^{- \gamma}\\
\sqrt{n}\{||\bm \beta_{j}^* + \frac{\bm u_{j}}{\sqrt{n}}||_{2}-||\bm \beta_{j}^*||_{2}\} &\leq & \sqrt{n}\{||\bm \beta_{j}^*||_{2}+\frac{1}{\sqrt{n}}||\bm u_{j}||_{2}-||\bm \beta_{j}^*||_{2}\} =||\bm u_{j}||_{2} < \infty \\
\sqrt{n}\{||\bm \beta_{j}^* + \frac{\bm u_{j}}{\sqrt{n}}||_{2}-||\bm \beta_{j}^*||_{2}\} &\geq & \sqrt{n}\{||\bm \beta_{j}^*||_{2}-||\frac{\bm u_{j}}{\sqrt{n}}||_{2}-||\bm \beta_{j}^*||_{2} \}=-||\bm u_{j} ||_{2} > -\infty
\end{eqnarray*}
This gives us 
\begin{eqnarray}
-||\bm u_{j}||_{2} \leq \sqrt{n}\{||\bm \beta_{j}^* + \frac{\bm u_{j}}{\sqrt{n}}||_{2}- ||\bm \beta_{j}^*||_{2}\} \leq ||\bm u_{j}||_{2}.	\label{bdd}
\end{eqnarray}
Using Slutsky we end up with 
\begin{eqnarray*}
	P_{1}^{(n)}&=& \lambda_{n}^1 \sum_{j=1}^J \left[ w_{1}^{(j)} ||\bm \beta_{j}^* +\frac{\bm u_{j}}{\sqrt{n}}||_{2} -w_{1}^{(j)}||\bm \beta_{j}^*||_{2}\right] \\	
	&=& \underbrace{\frac{\lambda_{n}^1}{\sqrt{n}}}_{\rightarrow 0} \sum_{j=1}^J \underbrace{w_{1}^{(j)}}_{\rightarrow_{p}||\bm \beta_{j}^*||_{2}^{-\gamma}} ||\bm \beta_{j}^*||^{-\gamma} \underbrace{\sqrt{n}\left[||\bm \beta_{j}^* +\frac{\bm u_{j}}{\sqrt{n}}||_{2} -||\bm \beta_{j}^*||_{2}\right]}_{\text{bounded using }(\ref{bdd})}\rightarrow_{p} 0
\end{eqnarray*}
for the case that $\bm \beta_{j}^* \neq \bm 0$. Now we come to the case that $\bm \beta_{j}^{*}= \bm 0$. In this case we get with $w_{1}^{(j)}= ||\tilde{\bm \beta}_{j}||_{2}^{-\gamma}= n^{\gamma/2}||\tilde{\bm \beta}_{j}\sqrt{n}||_{2}^{-\gamma}$
\begin{eqnarray*}
P_{1}^{(n)}= \lambda_{n}^1 \sum_{j=1}^J w_{1}^{(j)}|| \frac{\bm u_{j}}{\sqrt{n}}||_{2}= \frac{\lambda_{n}^1}{\sqrt{n}}\sum_{j=1}^J \sqrt{n} w_{1}^{(j)}||\frac{\bm u_{j}}{\sqrt{n}}||_{2}= \underbrace{\lambda_{n}^1n^{(\gamma-1)/2}}_{\rightarrow \infty}\sum_{j=1}^J \underbrace{||\tilde{\bm \beta_{j}}\sqrt{n}||_{2}^{-\gamma}}_{O_{p}(1)}\underbrace{||\bm u_{j}||_{2}}_{\text{bounded}}.	
\end{eqnarray*}
which goes to $\infty$ for $||\bm u_{j}||_{2} \neq 0$ and equals $0$ otherwise. Hence we get
\begin{eqnarray}
P_{1}^{(n)}	\rightarrow \begin{cases}
0  \,\,\,\,\,\, \text{if}\,\,\, ||\bm u_{j}||_{2}=0, \, \bm \beta_{j}^*= \bm 0,\\
\infty \,\,\,\, \text{if} \,\,\, ||\bm u_{j}||_{2} \neq 0, \, \bm \beta_{j}^* = \bm 0, \\
0 \,\,\,\,\,\,\,\, \text{if} \,\,\, \bm \beta_{j}^* \neq \bm 0 .	
 \end{cases} \label{a3.asy}
\end{eqnarray}
As in the proof of Theorem \ref{root-n-consistency} we know that
\begin{eqnarray}
P_{0}^{(n)}&&=\lambda_{n}^{0} \sum_{j=1}^J \sum_{0 \leq r < s \leq p_{j}}\left[w_{0}^{(j,rs)}||\beta_{j,r}^*-\beta_{j,s}^*+\frac{u_{j,r}-u_{j,s}}{\sqrt{n}}||_{0}-w_{0}^{(j,rs)}||\beta_{j,r}^*-\beta_{j,s}^*||_{0} \right] \nonumber \\
&&\leq \underbrace{\sum_{j=1}^J \sum_{0 \leq r < s \leq p_{j}} \,\,\,\underbrace{w_{0}^{(j,rs)} \lambda_{n}^0}_{\leq \, a_{n}^0\, \rightarrow \, 0\, }  \,p}_{\rightarrow\,\, 0}. \label{highdim2}
\end{eqnarray}
Consequently,
\begin{eqnarray}
	P_{0}^{(n)} \rightarrow_{p} 0.
\end{eqnarray}
 To sum up
\begin{eqnarray}
H_{n}(\bm u)= \underbrace{T_{1}^{(n)}}_{\rightarrow_{d} \bm u^T N(\bm 0, I(\bm \beta^{*}))}+\underbrace{T_{2}^{(n)}}_{\rightarrow_{p} \frac{1}{2} \bm u^T \bm I_{F}(\bm \beta^*) \bm u}+\underbrace{P_{1}^{(n)}}_{\text{see } (\ref{a3.asy})}+\underbrace{T_{3}^{(n)}}_{\rightarrow_{p} 0}+\underbrace{P_{0}^{(n)}}_{\rightarrow_{p} 0}	
\end{eqnarray}
holds which yields (note that for $j \in A$ it holds that $||\bm \beta_{j}^{*}||_{2} \neq \bm 0$, thus $\bm \beta_{j}^*\neq \bm 0$, and for $j \notin A$ it holds that $||\bm \beta_{j}^{*}||_{2} = \bm 0$ , thus $\bm \beta_{j}^{*} = \bm 0$, by definition of the active set $A$)
\begin{eqnarray*}
H_{n}(\bm u) \rightarrow_{d}H(\bm u)= \begin{cases}
 \bm u_{A}^T I_{11} \bm u_{A}-2 \bm u_{A}^T \bm W \,\,\,\, \text{if}\,\,\, u_{j}=0 \, \forall j \notin A\\
 \infty \hspace*{3.2cm} \text{otherwise}	
 \end{cases}
\end{eqnarray*}
where $\bm W = N(\bm 0, \bm I_{F}(\bm \beta^*))$. The minimum of $H(\bm u)$ is clearly at $(\bm I_{11}^{-1}\bm W_{A}, \bm 0)^T$ (the first part $\bm I_{11}^{-1}\bm W_{A}$ is for the indices $j \in A$ and the second part $\bm 0$ for the indices $j \notin A$). There exists $\hat{\bm u}$ satisfying $\hat{\bm u}_{A} \rightarrow_{d} I_{11}^{-1} \bm W_{A}$ and $\hat{\bm u}_{A^c} \rightarrow_{d} 0$. Now, because $\bm W_{A}=N(\bm 0, \bm I_{11})$, the asymptotic normality follows.	
\end{proof}
\begin{rem}[Difference to  conclusion in \cite{Zou2006}]
	Since $H_{n}(\bm u)$ is not convex because of the $L_{0}$ part in our approach, we can not conclude as in \cite{Zou2006} that $\hat{\bm u}_{A} \rightarrow_{d} I_{11}^{-1} \bm W_{A}$ and $\hat{\bm u}_{A^c} \rightarrow_{d} 0$ because we can not ensure that $H_{n}(\bm u)$ has a unique minimum. But we can say that there exists $\hat{\bm u}$ satisfying $\hat{\bm u}_{A} \rightarrow_{d} I_{11}^{-1} \bm W_{A}$ and $\hat{\bm u}_{A^c} \rightarrow_{d} 0$. Then, the asymptotic normality follows due to $\bm W_{A}=N(\bm 0, \bm I_{11})$.
\end{rem}

\subsection{Proof of Theorem \ref{sel.cons}}
\begin{proof}
The beginning of the proof follows \cite{Bunea2008} (Proof of Lemma 3.1) but we will transfer the proof to the more general case of $\bm \beta_{j}$ being a vector instead of a real number since we analyze categorical data. Having that, we will use the proven $\sqrt{n}$-consistency of our estimator (Theorem \ref{root-n-consistency}) to show the inequality. Note that, as already mentioned, the estimate $\hat{\bm \beta}$ depends on the sample size $n$, so in particular we have $\hat{\bm \beta}^{(n)}$. For simplicity, we leave out the upper index $(n)$ but keep in mind the dependence of the estimate on the sample size. We have that
\begin{eqnarray}
\mathbb{P}(A^* \not \subseteq A_{n}) && \leq \mathbb{P}(j \notin A_{n} \,\,\text{for some }\,j \in A^*)\nonumber \\
&&	\leq \mathbb{P}(\bm{\hat \beta}_{j}= \bm 0 \,\, \text{and}\,\, \bm \beta^*_{j} \neq \bm 0  \,\,\text{for some }\,j \in A^*) \nonumber\\
&& \leq \mathbb{P}(||\hat{\bm \beta}_{j}-\bm \beta^*_{j}||_{2}=||\bm \beta_{j}^*||_{2} \,\, \text{for some}\,\, j \in A^*) \nonumber \\
&& \leq \mathbb{P}(||\hat{\bm \beta}_{j}-\bm \beta_{j}^*||_{2} \geq \min_{l \in A^*}||\bm \beta_{l}^*||_{2} \,\, \text{for some}\,\, j \in A^*)\nonumber \\
&& \leq \mathbb{P}(||\hat{\bm \beta}-\bm \beta^*||_{2}\geq \min_{l \in A^*}||\bm \beta^*_{l}||_{2}).\label{star}
\end{eqnarray}
Note that $\min_{l \in A^*}||\bm \beta^*_{l}||_{2}$ is a minimum over a bounded set, since we assumed that the true underlying structure is sparse, thus the minimum always exists.
 Now our goal is to bound $(\ref{star})$ by some $\varepsilon$. Since we know from Theorem \ref{root-n-consistency} that $||\hat{\bm \beta}- \bm \beta^*||_{2}=O_{\mathbb{P}}(1/\sqrt{n})$ we get that $\forall \varepsilon >0$ there exists constants $M, \tilde{N} >0$ such that
\begin{eqnarray}
\mathbb{P}(||\sqrt{n}(\hat{\bm \beta}-\bm \beta^*)||_{2} > M) < \varepsilon\,\,\,\,\,\, \forall\, n\,>\,\tilde{N}.	\label{sel.cons.eq1}
\end{eqnarray}
Hence, for $n>\tilde{N}$ we have $\displaystyle \mathbb{P}\left(||\hat{\bm \beta}-\bm \beta^*||_{2} > \frac{M}{\sqrt{n}}\right) < \varepsilon$. Now, with $\varepsilon >0$ and constants $M,\tilde{N} > 0$, we can always choose some $N' >0$ such that $\displaystyle \frac{M}{\sqrt{n}} \leq \min_{l \in A^*}||\bm \beta_{l}^*||_{2}$ for all $n \geq N'$. Note that by definition we have that $||\bm \beta_{l}^*||_{2} \neq 0 \, \, \forall\, l \in A^*$.  Now we can write expression $(\ref{star})$ as
\begin{eqnarray}
\mathbb{P}(||\hat{\bm \beta}-\bm \beta^*||_{2}\geq \min_{l \in A^*}||\bm \beta^*_{l}||_{2}) \leq \mathbb{P}\left(||\hat{\bm \beta}-\bm \beta^*||_{2} > \frac{M}{\sqrt{n}}\right) < \varepsilon \,\,\, \forall\, n > \max\{\tilde{N},N'\}.
\end{eqnarray}
Consequently, $\forall \varepsilon > 0$ we can find some $N:=\max\{\tilde{N},N'\}$ such that 
\begin{eqnarray}
P(A^* \not \subseteq A_{n}) < \varepsilon \,\,\, \forall\, n \,>\, N 
\end{eqnarray}
which completes the proof.
\end{proof}

\subsection{Proof of Theorem \ref{sel.cons.div}}
\begin{proof}
	The proof works analogously to the proof of Theorem \ref{sel.cons} until we end up with (\ref{sel.cons.eq1}) which is modified using Theorem \ref{consistency-highdim}
	\begin{eqnarray}
\mathbb{P}(||\alpha_{n}^{-1}(\hat{\bm \beta}-\bm \beta^*)||_{2}> M) < \varepsilon\,\,\,\,\,\, \forall\, n\,>\,\tilde{N}.	
\end{eqnarray}
with $\alpha_{n}=\sqrt{\frac{p_{n}}{n}}$. The rest works analogously to Theorem \ref{sel.cons}.
\end{proof}

\section{Details on Approximation used in BCD}\label{app.approx}
For the approximation of the objective function $\tilde{g}(\bm \beta_{j},\bm\beta^{(k)})$ used in the covariate-wise BCD approach, we provide next the details on the derivation of the function $g(\bm \beta_{j},\bm\beta^{(k)})$, which is part of $\tilde{g}(\bm \beta_{j},\bm\beta^{(k)})$. In general, we use the following quadratic approximation of the $L_0$ part at some $\bm \beta^{(k)}$ (see also  \cite{OelkerEtAl2014})
\begin{eqnarray}
P_{\lambda}^{L_{0}}(\bm \beta) \approx P^{L_{0}}_{\lambda}(\bm \beta^{(k)}) +\frac{1}{2}(\bm \beta^T \bm A_{\lambda} \bm \beta + \bm \beta^{(k),T} \bm A_{\lambda} \bm \beta^{(k)}),\label{approxi1}
\end{eqnarray}
where details on the construction of $\bm A_\lambda$ can be found in \cite{OelkerEtAl2014}. For the covariate-wise approach we obtain $\bm A_{\lambda,j}$, on which details can be found in Remark \ref{rem.a.lambda}.
We proceed as follows: since our penalty function shows a separable structure, we obtain the solution coodinate-wise. With the help of a Taylor approximation of the log-likelihood, we approximate the $L_{0}$ penalty function $P_{\lambda}^{L_{0}}(\bm \beta_{j})$ separately for each $j \in \{1,...,J\}$ such that it is possible to follow a coordinate-wise procedure for minimization. So, we will obtain an approximation as in (\ref{approxi1}) for  $P_{\lambda}^{L_{0}}(\bm \beta_{j})$ for each $j \in \{1,...,J\}$.\\ 
Now it remains to obtain an approximation of the log-likelihood. In particular, we approximate the log likelihood with Taylor as in \cite{BrehenyHuang2011} yielding an approximation of $L_{n}(\bm \beta)$ given by 
\begin{eqnarray}
L_{n}(\bm \beta) \approx \frac{1}{2n}(\tilde{\bm y}- \bm X \bm \beta)^T \widetilde{\bm W} (\tilde{ \bm y} - \bm X \bm \beta).\label{approx.likeli}
\end{eqnarray}
Here, $\widetilde{\bm W}$ is a diagonal matrix of weights, see below for details.
\begin{rem}[On the factor of $\frac{1}{2n}$ in the log likelihood]
	In the literature, the log likelihood (or the squared difference in the linear model case respectively) is sometimes devided by the factor $\frac{1}{2n}$, as for example in \cite{BrehenyHuang2011} and \cite{GuoEtAl2015}. Since it does not change the solution of the minimum of the log Likelihood it is a convenient choice because it stabilizes the algorithm and it ensures that the impact of the tuning parameter $\lambda$ does not depend on the sample size $n$. Note that one can also neglect this factor but in this case one has to be careful when comparing two solutions for different tuning parmameters and different sample sizes respectively, but basically it works the same way. We will use this factor in the sections about computation with block coordinate descent (BCD) and keep in mind that it is not used by \cite{OelkerEtAl2014} in PIRLS, even though there is an option in the package \texttt{gvcm.cat} that divides the log-likelihood by $n$.
\end{rem}
For the matrix with weights it holds 
\begin{eqnarray}
\widetilde{\bm W}= \text{diag}(w_{i}) \in \mathbb{R}^{n \times n},\,\,w_{i}=\pi_{i}(1-\pi_{i}) \,\,\, \text{for}\,\, i=1,...,n
\end{eqnarray}
where $\displaystyle \pi_{i}= \frac{\exp(\eta_{i})}{1+\exp(\eta_{i})}$ and $\eta_{i}=(\bm X \bm \beta^{(k)})_{i}$ thus $\pi_{i}$ is evaluated at the current iteration $k$. Thus, $\widetilde{\bm W}$ also depends on the iteration step and should be denoted by $\widetilde{\bm W}^{(k)}$ but for simplicity we omit the index $k$. 

\begin{rem}[On notation] Note that $\bm \mu = \bm \pi$ (in particular it holds that $\frac{\partial \mu}{\partial \eta}=\pi(1-\pi)=\frac{\exp(\eta)}{(1+\exp(\eta))^2}$). 
\end{rem}
Furthermore, the working response $ \tilde{\bm y}$ is given by 
\begin{eqnarray}
\tilde{\bm y}=\bm X \bm \beta^{(k)}+\widetilde{\bm W}^{-1}(\bm y- \bm \pi)		
\end{eqnarray}
where, $\bm \pi=(\pi_{1},...,\pi_{n})$ is evaluated at $\bm \beta^{(k)}$. For each $j \in \{1,....,J\}$ we get the following approximation of the penalty term in $\bm \beta_{j}^{(k)}$, analogously to (\ref{approxi1})
\begin{eqnarray}
	P_{\lambda}^{L_{0}}(\bm \beta_{j})\approx P_{\lambda}^{L_{0}}(\bm \beta_{j}^{(k)})+\frac{1}{2}(\bm \beta_{j}^T\bm A_{\lambda,j}\bm \beta_{j}+(\bm\beta_{j}^{(k)})^T \bm A_{\lambda,j}\bm \beta_{j}^{(k)}).\label{approx.pen}
\end{eqnarray}
Note that, in particular one has to write $\bm A_{\lambda,j}^{(k)}$ instead of $\bm A_{\lambda,j}$ since this quantity depends on the iteration step $k$, but we will leave the upper index out for simplicity.
\begin{rem}\label{rem.a.lambda}[Details on $ \bm A_{\lambda,j}$ for $L_{0}$]
For covariate $j \in \{1,...,J\}$ with $p_{j}$+1 levels (including the reference category), the components of the approximation look as follows where we assume observing a nominal covariate. For an ordinal one the value for $|L_{j}|$ will change as we just compare adjacent categories for ordinal factors. Let $L_{j}$ be the set containing the row numbers of the matrix $\bm A$ with rows $\bm a_{l}$ that correspond to the differences for covariate $j$. We have with $\bm a_{l,j}\,(l \in \{1,...,|L_{j}|\})$ being the columns of some matrix $\bm A_{j}$ (not to be mixed up with $\bm A_{\lambda,j}$) that produces the differences of the entries in $\bm \beta_{j}^{(k)}$ that
\begin{eqnarray}
|L_{j}| &&= p_{j}+ {p_{j}\choose 2}=\frac{p_{j}(p_{j}-1)}{2} \hspace*{2cm}\text{(number of differences of entries in $\bm \beta_{j}^{(k)}$)} \nonumber \\
\bm A_{\lambda,j} &&=\lambda_{0}\sum_{l =1}^{|L_{j}|} p_{l}'(||\bm a_{l,j}^T \hat{\bm \beta}^{(k)}_{j}||_{0})\frac{D_{l}(\bm a_{l,j}^T \hat{\bm \beta}^{(k)}_{j})}{\bm a_{l,j}^T \hat{\bm \beta}^{(k)}_{j}} \bm a_{l,j}\bm a_{l,j}^T \nonumber \\
&&= \lambda_{0}\sum_{l=1}^{|L_{j}|} \left(\frac{1}{1+\exp(-\gamma|\bm a_{l,j}^T \bm \hat{\bm \beta}^{(k)}_{j}|)}\right)\left(1-\frac{1}{1+\exp(-\gamma|\bm a_{l,j}^T \hat{\bm \beta}^{(k)}_{j}|)}\right)\nonumber  \\
&& \,\,\,\,\,\,\,\,\,\,\,\,\,\,\,\,\,\,\,\,\,\,\,\, \cdot \frac{2 \gamma \bm a_{l,j} \bm a_{l,j}^T}{\sqrt{(\bm a_{l,j}^T \hat{\bm\beta}^{(k)}_{j})^2+c}}\label{a.lambda}
\end{eqnarray}
We have that the columns $\bm a_{l,j} \in \mathbb{R}^{p_{j}\times 1}$ so they produce the differences of the coefficients and since they also include a columns in the shape of $(1,0,...,0),(0,1,0,....,0), ...,(0,...,0,1 )$, so just one entry equal to one and the others equal to zero, they also build the differences of each coefficient with reference category zero. It holds that $\bm A_{\lambda,j} \in \mathbb{R}^{p_{j}\times p_{j}}$ and $\bm A_{\lambda,j}$ is symmetric. Note that $\bm A_{\lambda,j}$ does not depend on $\bm \beta_{j}$, it just depends on $\hat{\bm \beta}^{(k)}_{j}$.
\end{rem}
So the following function $g(\bm \beta_{j},\bm\beta^{(k)})$ will be the covariate-wise approximation of the log likelihood and $L_0$ penalty part that we use in the BCD procedure
\begin{eqnarray}
	g(\bm \beta_{j}, \bm \beta^{(k)}):=&&\frac{1}{2n}(\tilde{\bm y}- \bm X \bm \beta)^T \widetilde{\bm W} (\tilde{ \bm y} - \bm X \bm \beta)+ P^{L_{0}}_{\lambda}(\bm \beta^{(k)}_{j}) \nonumber \\ && + \frac{1}{2} (\bm \beta^T_{j} \bm A_{\lambda,j} \bm \beta_{j} + \bm \beta^{(k),T}_{j}\bm A_{\lambda,j} \bm \beta^{(k)}_{j}). \label{defg}
\end{eqnarray}

\section{Details on Simulation Study}\label{app.sim}
\subsection{Details on Tuning}
We used CV for all penalties to determine the tuning parameter $\lambda_{0}$ or $\bm \lambda =(\lambda_{0} ,\lambda_{1})$ for the $L_{0}$-FGL approach. In particular, we used $k=5$ fold CV, where we used $\lambda_{lower}=0$ and for $\lambda_{upper}$ we chose a value which excludes all variables from the model. Note that for $L_{0}$
-FGL we need two maximum values for lambda, hence $\lambda_{max}=(\lambda_{max,1},\lambda_{max,0})$. So for the case of two tuning parameters, we chose them in a way such that for $\lambda_{max}=(\lambda_{max,1},\lambda_{max,0})$ all parameters are excluded from the model where we took $\lambda_{max,1}=\lambda_{max,0}$ to avoid that we set the focus on selection or fusion. Between these two values, the CV procedure fitted the model for $n_{\lambda}=10$ different values of $\lambda_{0}$. For the CV of $L_{0}$ with PIRLS, we used the stored functions in $\texttt{gvcm.cat}$. As explained at the beginning of this work, for the cross validation procedure for $L_{0}$-FGL, which includes two tuning parameters, we chose a two step procedure. In particular, we first set the tuning of the $L_{0}$ part to zero ($\lambda_0=0$) and determine the optimal value for the group lasso part ($\lambda_{1,\text{opt}}$). Then, fixing the tuning of the group lasso part to this optimal value ($\lambda_{1,\text{opt}}$), we determine the best value for the tuning parameter of the $L_{0}$ part, resulting in ($\lambda_{0,\text{opt}}$). Here, the "best" model is chosen wrt the predictive deviance measure. The parameters for the approximation of the $L_{0}$ part, which is used in all of our considered methods, where chosen equally in all approaches $c=10^{-5}$ and $\gamma=10$ (recommended in \cite{OelkerEtAl2014}). Even if $L_{0}$-FGL with BCD do not require a stepsize, we used a stepsize of $\nu=0.05$ for all considered approaches. This is done to stabilize the algorithm and to obtain comparable results.
 
\end{document}